\documentclass[12pt,twoside,reqno]{amsart}
\usepackage{amsfonts}
\usepackage{color}
\usepackage{mathrsfs}
\usepackage{stmaryrd}
\usepackage{graphicx,psfrag,epsfig}
\newcommand{\RR}{\mathbb R} 
\newcommand{\R}{\mathbb R}

\newcommand{\rd}{\mathrm{d}}
\newcommand{\beqn}{\begin{equation}}
\newcommand{\eeqn}{\end{equation}}
\newcommand{\bean}{\begin{eqnarray}}
\newcommand{\eean}{\end{eqnarray}}
\DeclareMathAlphabet{\mathpzc}{OT1}{pzc}{m}{it}
\newtheorem{theorem}{Theorem}[section]

\newtheorem{lemma}[theorem]{Lemma}
\newtheorem{proposition}[theorem]{Proposition}
\newtheorem{definition}[theorem]{Definition}
\newtheorem{remark}[theorem]{Remark}
\newtheorem{remarks}[theorem]{Remarks}

\numberwithin{equation}{section}
\begin{document}
\title{Bounded traveling waves for a thin film with gravity\\ and insoluble surfactant} 
\thanks{This work was partially supported by the french-german PROCOPE project 20190SE}
\author{Joachim Escher}
\address{Leibniz Universit\"at Hannover, Institut f\"ur Angewandte Mathematik, Welfengarten 1, D--30167 Hannover, Germany}
\email{escher@ifam.uni-hannover.de}
\author{Matthieu Hillairet}
\address{CEREMADE, UMR CNRS~7534, Universit\'e de Paris-Dauphine \\ Place du Mar\'echal De Lattre De Tassigny,
F--75775 Paris Cedex 16, France}
\email{hillairet@ceremade.dauphine.fr}
\author{Philippe Lauren\c cot}
\address{Institut de Math\'ematiques de Toulouse, CNRS UMR~5219, Universit\'e de Toulouse, F--31062 Toulouse cedex 9, France} 
\email{laurenco@math.univ-toulouse.fr}
\author{Christoph Walker}
\address{Leibniz Universit\"at Hannover, Institut f\"ur Angewandte Mathematik, Welfengarten 1, D--30167 Hannover, Germany}
\email{walker@ifam.uni-hannover.de}


\date{\today}

%%%%%%%%%%%%%%%%%%%%%%%%%%%%%%%%%%%%%%%%%%%%%%%%%%%%%%%%%%%%%%%%%%%%
\begin{abstract}
Lubrication equations for a surfactant-driven flow of a thin layer of fluid are considered with a singular surfactant-dependent surface tension. It is shown that there exists a bounded traveling wave solution which connects a fully surfactant-coated film to an uncoated film. The regularity of the traveling wave depends on whether or not surface diffusion of the insoluble surfactant is included.
\end{abstract}
%%%%%%%%%%%%%%%%%%%%%%%%%%%%%%%%%%%%%%%%%%%%%%%%%%%%%%%%%%%%%%%%%%%%

\maketitle

%
%     HEADLINES
%
\pagestyle{myheadings}
\markboth{\sc{Escher, Hillairet, Lauren\c{c}ot $\&$ Walker}}{\sc{Traveling waves in thin films with insoluble surfactant}}

%%%%%%%%%%%%%%%%%%%%%%%%%%%%%%%%%%%%%%%%%%%%%%%%%%%%%%%%%%%%%%%%%%%%
%%%%%%%%%%%%%%%%%%%%%%%%%%%%%%%%%%%%%%%%%%%%%%%%%%%%%%%%%%%%%%%%%%%%
\section{Introduction and Main Results}
%%%%%%%%%%%%%%%%%%%%%%%%%%%%%%%%%%%%%%%%%%%%%%%%%%%%%%%%%%%%%%%%%%%%
%%%%%%%%%%%%%%%%%%%%%%%%%%%%%%%%%%%%%%%%%%%%%%%%%%%%%%%%%%%%%%%%%%%%

The study of the dynamical behavior of viscous thin films coating surfaces is a classical topic in fluid dynamics. A widely used approach in this field is to derive simplified model equations from lubrication theory rather than studying the full equations of fluid mechanics based on first principles. The fact that surface tension effects become significant, or even dominant, is a common structural feature of many of these models. Surface tension itself is very sensitive to surface active agents, so-called surfactants, on the free surface of the thin film. Surfactant-driven flows of thin liquid films actually have attracted considerable interest in recent years.

In this paper, we study a model which describes an infinitely extended one-dimensional monolayer thin film carrying an insoluble surfactant. This model is essentially due to Jensen and Grotberg, see \cite{JenGro92, JenGro93}.  It involves besides surface diffusion of the surfactant also gravitational effects but neglects fourth-order capillarity terms. Writing $h(t,x)$ and $\gamma(t,x)$ for the film height and the surfactant concentration, respectively, at the time $t>0$ and the position $x\in\mathbb R$, the system reads as:
\begin{eqnarray}
\partial_t h + \partial_x \left( - \frac{G h^3}{3} \partial_x h + \frac{h^2}{2} \sigma'(\gamma) \partial_x\gamma \right) & = & 0 \ , \qquad  (t,x)\in (0,\infty) \times \mathbb R\ , \label{i1} \\
\partial_t\gamma + \partial_x \left( - \frac{G h^2}{2} \gamma \partial_x h + \left[ h \gamma \sigma'(\gamma) - D \right] \partial_x\gamma \right) & = & 0 \ , \qquad (t,x)\in (0,\infty) \times \mathbb R\ , \label{i2}
\end{eqnarray}
after a suitable scaling so that the non-negative parameters $D$ and $G$ representing a surface diffusion coefficient and a gravitational force, respectively, are non-dimensional. The function $\sigma$ denotes the surfactant-dependent surface tension. A common choice of $\sigma$ used in the literature is $\sigma_\infty(\gamma) := \sigma_0(1 - \gamma$) with $\sigma_0>0$, which is the limit as $\beta\to\infty$ of the Sheludko equation
\begin{equation}
\sigma_\beta(\gamma) := \sigma_0 \beta \left[ \left( \frac{1 + \theta(\beta)}{1 + \theta(\beta) \gamma} \right)^{3} - 1 \right]\ , \quad \theta(\beta) := \left( \frac{\beta+1}{\beta} \right)^{1/3} - 1\ , \quad \beta>0\ , \label{JG}
\end{equation}
see \cite{CMP09, JenGro92}. Other choices of surface tensions in the literature include the Szyszkowski equation $\sigma_{Sz}$ and the Frumkin equation $\sigma_{Fr}$
\begin{equation}
\sigma_{Sz}(\gamma) := \sigma_0 + a \ln{(1-\gamma)}\ , \quad \sigma_{Fr}(\gamma) := \sigma_0 + a \ln{(1-\gamma)} + b \gamma^2\ , \label{SzFr}
\end{equation}
for positive constants $\sigma_0, a, b$ \cite{CF95, CMP09}. A noticeable difference between the family  $(\sigma_\beta)_{\beta>0},$ on the one hand, and the pair  $\sigma_{Sz}, \sigma_{Fr},$ on the other hand, is that the latter constraints the (rescaled) surfactant concentration to take values in $[0,1)$ as expected on physical grounds. We refer to \cite{BN04, BGN03, CT13, EHLW11, EHLW12a, EHLW12b, GW03, Renardy96a, Renardy96b, Renardy97} and the references therein regarding numerical and well-posedness results for equations \eqref{i1}-\eqref{i2} and variants thereof.

\medskip

In this paper we are looking for traveling wave solutions to \eqref{i1}-\eqref{i2}, this means solutions of the form
\begin{equation}
(h,\gamma)(t,x) = (H,\Gamma)(x-ct)\ , \qquad (t,x)\in  (0,\infty)\times\RR\ , \label{i4}
\end{equation}
with velocity $c\in\RR$, $c\ne 0$. We focus on uniformly bounded traveling waves connecting a fully coated state ($\Gamma\sim 1$) to an uncoated state ($\Gamma\sim 0$), a situation corresponding to the spreading of the surfactant on the thin film. Substituting the ansatz \eqref{i4} in \eqref{i1}-\eqref{i2}, the pair of profiles $(H,\Gamma)$ is governed by
\begin{eqnarray}
c H + \frac{G H^3}{3} H' - \frac{H^2}{2} \sigma'(\Gamma) \Gamma' & = & K_1\ , \qquad \xi\in \RR\ , \label{ii5} \\
c \Gamma + \frac{G H^2}{2} \Gamma H' - \left[ H \Gamma \sigma'(\Gamma) - D \right] \Gamma' & = & K_2 \ , \qquad \xi\in \RR\ , \label{ii6}
\end{eqnarray}
for some $(K_1,K_2)\in \RR^2$ with the new variable $\xi:= x-ct$. The number of unknown parameters in \eqref{ii5}-\eqref{ii6} can be reduced after noticing that, if $(H,\Gamma)$ solves \eqref{ii5}-\eqref{ii6} for the parameters $(c,K_1,K_2)$ with $c\ne 0$, then $\xi\mapsto (H,\Gamma)(\xi/c)$ solves \eqref{ii5}-\eqref{ii6} for the parameters $(1,K_1/c,K_2/c)$. We may thus assume without loss of generality that $c=1$ throughout the paper.

\medskip
 
In contrast to previous investigations where traveling wave solutions are constructed for surfactant-driven flows of thin liquid films down an inclined plane with surface tension $\sigma_\infty$ depending linearly on the surfactant concentration \cite{LS06, LSW07, MS09, WSL06}, the thin film is here assumed to lie above a horizontal support and our aim is to consider the influence of a singularity in the surface tension such as \eqref{SzFr}. More precisely, we shall rather suppose herein that the surface tension $\sigma$ enjoys the following properties:
\begin{equation}
\sigma\in  C^1([0,1)) \;\text{ is a decreasing function with } \; \sigma'\not\in L_1(0,1)
\;\text{ and }\; \sigma_0 := \sigma(0) >0 \ . \label{i3}
\end{equation}
Note that $\sigma$ then realizes a non-increasing one-to-one mapping between $[0,1)$ and 
$(-\infty,\sigma_0]$ with
$$
\sigma' (z)\rightarrow -\infty\ \text{ as }\ z\rightarrow 1\, .
$$
Thus, letting 
$$
\varrho(z) := -1/\sigma'(z)\ , \quad  z \in [0,1)\ , \;\text{ and }\; \varrho(1) := 0\ ,
$$
we have that
\begin{equation}
\varrho \in C([0,1] , [0,\infty))\ \text{vanishes only at}\ z=1  \;\;\;\text{ and we put }\;\; R_\varrho := \|\varrho\|_{L_\infty(0,1)}\ . \label{rho1} 
\end{equation}
If $G>0$ we further assume that
\begin{equation}
\varrho\in \mathrm{Lip}([0,1))\ , \label{assumrho}
\end{equation}
that is, $\varrho$ is a {\em locally} Lipschitz continuous function on $[0,1)$ with local Lipschitz constant $L_z$ on $[0,z]$ possibly tending to $\infty$ as $z\to 1$. We point out that the Szyszkowski surface tension $\sigma_{Sz}$ and the Frumkin surface tension $\sigma_{Fr}$ defined in \eqref{SzFr}, with parameters $a$ and $b$ such that $b<2a$, satisfy \eqref{i3}-\eqref{assumrho}.

\medskip

With these assumptions, our main result reads:

%%%%%%%%%%%%%%%%%%%%%%%%%%%%%%%%%%%%%%%%%%%%%%%%%%%%%%%%%%%%%%%%%%%%
\begin{theorem}\label{T1}
Assume that the surface tension $\sigma$ satisfies \eqref{i3} and that the function $\varrho$, defined in \eqref{rho1}, satisfies \eqref{assumrho} when $G>0$. Given $H_*>0$, there is at least one traveling wave solution to \eqref{i1}-\eqref{i2} with velocity $c=1$ and profile $(H,\Gamma)$ such that
\begin{align*}
&\lim_{\xi\to-\infty} H(\xi)=2H_*\ , & & \lim_{\xi\to\infty} H(\xi)= H_*\ ,\\
&\lim_{\xi\to-\infty} \Gamma(\xi)=1 \ , & & \lim_{\xi\to\infty} \Gamma(\xi)=0\ ,
\end{align*}
and $\Gamma$ is non-increasing. 
\end{theorem}
%%%%%%%%%%%%%%%%%%%%%%%%%%%%%%%%%%%%%%%%%%%%%%%%%%%%%%%%%%%%%%%%%%%%

Let us point out that \eqref{ii5}-\eqref{ii6} is a system of ordinary differential equations for $(H,\Gamma)$ only when $G>0$, as the derivative of $H$ is no longer involved in the absence of gravity, $G=0$. Also, if surface diffusion is neglected, $D=0$, then \eqref{ii6} is independent of the derivative of $\Gamma$ where $H$ vanishes.

\medskip

Further properties and the regularities of $H$ and $\Gamma$ thus vary according to whether or not $G$ or $D$ are positive. We shall gather them in the following proposition:

%%%%%%%%%%%%%%%%%%%%%%%%%%%%%%%%%%%%%%%%%%%%%%%%%%%%%%%%%%%%%%%%%%%%
\begin{proposition}\label{P222} Let $c=1$ and $H_*>0$.

\noindent (i) If both gravity and diffusion are neglected, that is, if $G=0$ and $D=0$, then $(H ,\Gamma)\in C(\R,\R^2)$ is continuously differentiable except at $\xi=0$ and  given by
\begin{equation}\label{cc1}
H(\xi) = \left\{
\begin{array}{ccl}
2H_* \ , & \xi\in (-\infty,0]\ , \\
 &  \\
 H_* \ , & \xi\in [0,\infty)\ ,
\end{array}
\right.
 \qquad 
\Gamma(\xi) = \left\{
\begin{array}{ccl}
\displaystyle{\sigma^{-1}\left( \sigma_0 + \frac{\xi}{2H_*} \right)}\ , &  \xi\in (-\infty,0]\ , \\
 &  \\
 0\ , & \xi\in [0,\infty)\ .
\end{array}
\right.
\end{equation}

\noindent (ii) If gravity is neglected but diffusion is taken into account, that is, if $G=0$ and $D>0$, then $H\in C(\R)$ is given by \eqref{b4a} (with $\mathcal{P}_\Gamma=\R$) while $\Gamma\in C^1(\R)$ is implicitly given by \eqref{210} and $(H,\Gamma)$ satisfies
 $$
H_* < H(\xi) < 2 H_* \ ,\quad 0< \Gamma(\xi)< 1\ ,\qquad \xi\in \mathbb{R}\ .
$$

\noindent (iii) If gravity is taken into account, but diffusion is neglected, that is, if $G>0$ and $D=0$, then
$(H,\Gamma)\in C(\R,\R^2)$ is continuously differentiable except at {$\xi = 0$}. Moreover, $H$ is decreasing on $[0,\infty)$ and equals $2H_*$ on $(-\infty,0]$, while $\Gamma\equiv 0$ on $[0,\infty)$.

\noindent (iv) If both gravity and diffusion are taken into account, that is, if $D>0$ and $G>0$, then  $(H,\Gamma)\in C^1(\R,\R^2)$ satisfies
$$
H_* < H(\xi) < 2 H_* \ , \quad 0< \Gamma(\xi)< 1\ ,\qquad \xi\in \mathbb{R}\ .
$$
\end{proposition}
%%%%%%%%%%%%%%%%%%%%%%%%%%%%%%%%%%%%%%%%%%%%%%%%%%%%%%%%%%%%%%%%%%%%

In all cases, the profile $(H,\Gamma)$ is not regular enough and $(h,\gamma)$, defined through \eqref{i4}, does not satisfy \eqref{i1}-\eqref{i2} in a classical sense. However, it is a weak solution in a sense made precise in Definition~\ref{def_ws} below, which is adapted from the definition introduced in \cite{EHLW11}.

\medskip

The outline of the paper is as follows.  In the next section, besides providing system \eqref{i1}-\eqref{i2} with a definition of weak solutions, we state some elementary facts on traveling waves. In Section~\ref{sec_legere} we first  construct traveling waves when gravity is neglected, $G=0$. In Section~\ref{sec_grave} we then focus on the case with $G>0$. 
Finally, in Section~\ref{stat_sol} we briefly discuss the construction of stationary solutions, that is, the case $c=0$.

%%%%%%%%%%%%%%%%%%%%%%%%%%%%%%%%%%%%%%%%%%%%%%%%%%%%%%%%%%%%%%%%%%%%
%%%%%%%%%%%%%%%%%%%%%%%%%%%%%%%%%%%%%%%%%%%%%%%%%%%%%%%%%%%%%%%%%%%%
\section{Generalities on \eqref{i1}-\eqref{i2} and traveling waves}\label{sec_general}
%%%%%%%%%%%%%%%%%%%%%%%%%%%%%%%%%%%%%%%%%%%%%%%%%%%%%%%%%%%%%%%%%%%%
%%%%%%%%%%%%%%%%%%%%%%%%%%%%%%%%%%%%%%%%%%%%%%%%%%%%%%%%%%%%%%%%%%%%

Motivated by the option to allow for non-smooth traveling wave solutions for \eqref{i1}-\eqref{i2}, we introduce the following notion of weak solutions to system \eqref{i1}-\eqref{i2}.

%%%%%%%%%%%%%%%%%%%%%%%%%%%%%%%%%%%%%%%%%%%%%%%%%%%%%%%%%%%%%%%%%%%%
\begin{definition} \label{def_ws}
Given a bounded open interval $I$, a pair $(h,\gamma)\,:\, I\times\mathbb{R}\to \mathbb{R}^2$ is called a {\em{weak solution}} to \eqref{i1}-\eqref{i2} on $I$, if
\begin{itemize}
\item[(i)] $h\in L_{\infty}(I,L_{2,loc}(\mathbb{R})) \cap L_2(I,L_{\infty,loc}(\mathbb R))$ with $G h^{5/2} \in L_2(I , W^{1}_{2,loc}(\mathbb{R}))$ satisfies
$$
h(t,x) \geq 0 \quad \text{ for a.a. $(t,x) \in I \times \mathbb{R}$}\,,   
$$ 
\item[(ii)] $\gamma \in L_{\infty}(I\times \mathbb{R})$ with $D \gamma \in L_2(I,W^{1}_{2,loc}(\mathbb{R}))$ satisfies
$$
\gamma(t,x) \in [0,1) \quad \text{ for a.a. $(t,x) \in I \times \mathbb{R}$,}
$$
\item[(iii)] $\sigma(\gamma) \in L_2(I,W^1_{2,loc}(\mathbb{R})),$   
\item[(iv)] for all $\varphi \in C_0^\infty(I \times \mathbb{R}),$ there holds 
\begin{eqnarray}
\int_{I} \int_{\mathbb{R}}   \left(h\, \partial_t \varphi + \dfrac{h^2}{2} \partial_x \sigma(\gamma) \, \partial_x\varphi - \dfrac{G}{12} \partial_x h^4 \, \partial_x \varphi \right) {\rm d}x{\rm d}t &=& 0\,, \label{eq_i1weak}\\
\int_{I} \int_{\mathbb R}   \left( \gamma\partial_t \varphi + h \gamma  \partial_x \sigma(\gamma) \, \partial_x \varphi - D \partial_x \gamma \, \partial_x \varphi -  \dfrac{G}{6} \gamma \partial_x h^3 \, \partial_x\varphi\right) {\rm d}x{\rm d}t &=& 0\,. \label{eq_i2weak}
\end{eqnarray}
\end{itemize}
Given an unbounded open interval $I,$ a pair $(h,\gamma)$ is a weak solution to \eqref{i1}-\eqref{i2} on $I$, if it is 
a weak solution to \eqref{i1}-\eqref{i2} on any bounded open subinterval of $I$.
\end{definition}
%%%%%%%%%%%%%%%%%%%%%%%%%%%%%%%%%%%%%%%%%%%%%%%%%%%%%%%%%%%%%%%%%%%%

%%%%%%%%%%%%%%%%%%%%%%%%%%%%%%%%%%%%%%%%%%%%%%%%%%%%%%%%%%%%%%%%%%%%
\begin{remarks} {\bf (a)} {\rm Definition~\ref{def_ws} is reminiscent of \cite[Theorem 1.1]{EHLW11}, where system \eqref{i1}-\eqref{i2} is considered on a finite  spatial interval, complemented with homogeneous Neumann boundary conditions. The situation studied in \cite{EHLW11} is, however, slightly different as both parameters $D$ and $G$ are strictly positive and the surface tension $\sigma$ grows at most linearly in $\gamma$.} \\[0.7ex]
{\bf (b)} {\rm In \cite[Theorem 1.1]{EHLW11} the regularity of weak solutions is included in their construction. The main tool to guarantee this regularity is an {\em a priori} estimate based on an energy functional, established first for strictly positive solutions in \cite[Proposition 3.2]{EHLW12a} and then extended to non-negative solutions in \cite[Lemma 2.7]{EHLW11}.
Roughly speaking, we ask here for the regularity inherited by such weak solutions also in case when $G$ or $D$ vanish. It should be noted that the energy estimate is no longer available if $G=0$.} \\[0.7ex]
{\bf (c)} {\rm Recall that $\sigma'(z)$ blows up at $z = 1.$ Hence, condition~(iii) implies formally that $\gamma$ must remain below $1$ almost everywhere. We introduce this explicitly in condition~(ii).} \\[0.7ex]
{\bf (d)} {\rm Finally, the weak formulation \eqref{eq_i1weak}-\eqref{eq_i2weak} is obtained from \eqref{i1}-\eqref{i2} by assuming that the solution $(h,\gamma)$  is sufficiently smooth, multiplying by a test function $\varphi$ and integrating by parts. In particular, when $D>0$, if a pair $(h,\gamma)$ with regularity
\begin{eqnarray}
 h \in C^1(I\times\mathbb{R}) \,,\quad \gamma \in C^{2}(I\times\mathbb{R}) && \text{ if }\ G=0\ , \\
 h \in C^2(I\times\mathbb{R}) \,,\quad \gamma \in C^{2}(I\times\mathbb{R}) && \text{ if }\ G>0\ ,
 \end{eqnarray}
satisfies equations \eqref{i1}-\eqref{i2} pointwisely together with $\gamma(t,x) < 1$ for all $(t,x)$ in $I\times\mathbb{R}$, then it is a weak solution to \eqref{i1}-\eqref{i2} on $I$ in the sense of Definition~\ref{def_ws}.}
\end{remarks}
%%%%%%%%%%%%%%%%%%%%%%%%%%%%%%%%%%%%%%%%%%%%%%%%%%%%%%%%%%%%%%%%%%%%
 
As noted in the introduction, the profile $(H,\Gamma)$ of a traveling wave solution satisfies \eqref{ii5}-\eqref{ii6} in the new variable $\xi= x-ct$ for some $c\in\R$ and $(K_1,K_2)\in \RR^2$ and we recall that we may assume without loss of generality that $c=1$ if $c\ne 0$. Since we are mainly interested in traveling waves connecting a fully coated state ($\Gamma\sim 1$) to an uncoated state ($\Gamma\sim 0$), we expect that $\Gamma$ decays to zero as $\xi\to\infty$ and that $\Gamma'$ decays to zero as $\xi\to\pm \infty$. Inserting these properties in \eqref{ii6} and assuming that $H$ is bounded and $H'$ decays to zero as $\xi\to\pm \infty$, we conclude that $K_1=0$ if $c=0$ and that $K_2=0$ whatever the value of $c$. The case $c=0$ corresponds to stationary solutions which turn out to be uninteresting from a physical viewpoint. For this reason we focus our attention on the case $c\ne 0$ and postpone the discussion of $c=0$ to Section~\ref{stat_sol}. Consequently, setting $H_*:=K_1$, we look for global solutions $(H,\Gamma)$ to the following system
\begin{align}
H + \frac{G H^3}{3} H' - \frac{H^2}{2} \sigma'(\Gamma) \Gamma' & = H_*\ , &&\xi\in \RR\ , \label{i5} \\
\Gamma + \frac{G H^2}{2} \Gamma H' - \left[ H \Gamma \sigma'(\Gamma) - D \right] \Gamma' & =  0 \ ,  &&\xi\in \RR\ , \label{i6}\\
H\ge 0\ ,\qquad 0\le \Gamma &< 1\ , &&\xi\in \RR\ , \label{i6i}
\end{align}
satisfying in addition
\begin{equation}
\lim_{\xi\to -\infty} \Gamma(\xi) = 1\ , \qquad \lim_{\xi\to \infty} \Gamma(\xi) = 0\ , \label{limGam}
\end{equation} 
where $H_*$ is a non-negative parameter yet to be determined.  Clearly, for $H_*\ge 0$, a solution to \eqref{i5}-\eqref{i6i} is given by $(H,\Gamma)=(H_*,0)$, but it does not satisfy \eqref{limGam}.

\medskip

Solutions to \eqref{i5}-\eqref{i6i} generate traveling wave (weak) solutions to \eqref{i1}-\eqref{i2} as we show now.

%%%%%%%%%%%%%%%%%%%%%%%%%%%%%%%%%%%%%%%%%%%%%%%%%%%%%%%%%%%%%%%%%%%%
\begin{lemma}\label{TWws}
Let $(H,\Gamma)$ be a global solution to \eqref{i5}-\eqref{i6i} such that $\Gamma$ is continuous and piecewise $C^1$-smooth and either $H$ is piecewise continuous if $G=0$ or $H^{5/2}$ is continuous and piecewise $C^1$-smooth if $G>0$. Then $(h,\gamma)$ defined on $\RR\times\RR$ by $(h,\gamma)(t,x):=(H,\Gamma)(x-t)$ is a weak solution to \eqref{i1}-\eqref{i2} on $\R$.
\end{lemma}
%%%%%%%%%%%%%%%%%%%%%%%%%%%%%%%%%%%%%%%%%%%%%%%%%%%%%%%%%%%%%%%%%%%%

\begin{proof}
The conditions (i)-(iii) of Definition~\ref{def_ws} being clearly satisfied, it remains to check (iv). To this end, we note that, given $\varphi\in C_0^\infty(\R\times \R)$, 
\begin{align*}
\int_{\R\times\R} h(t,x) \partial_t \varphi(t,x)\ \mathrm{d}(x,t) = & \int_{\R\times\R} H(\xi) \partial_t \varphi(t,\xi+t)\ \mathrm{d}(\xi,t) \\
= & \int_{\R} H(\xi) \int_{\R} \frac{\mathrm{d}}{\mathrm{d}t} \left[ \varphi(t,\xi+t) \right]\ \mathrm{d}t\ \mathrm{d}\xi -  \int_{\R\times\R} H(\xi) \partial_x \varphi(t,\xi+t)\ \mathrm{d}(\xi,t) \\ 
= & -  \int_{\R\times\R} h(t,x) \partial_x \varphi(t,x)\ \mathrm{d}(x,t)\ .
\end{align*}
The left-hand side of \eqref{eq_i1weak} then reads
\begin{align*} 
& \int_{\R\times\R} \left( - h + \frac{h^2}{2} \partial_x \sigma(\gamma) - \frac{G}{12} \partial_x h^4 \right)(t,x) \partial_x\varphi(t,x)\ \mathrm{d}(x,t)\\
= & \int_{\R\times\R} \left( - H + \frac{H^2}{2} \sigma'(\Gamma) \Gamma' - \frac{G}{12}  (H^4)' \right)(\xi)\ \partial_x\varphi(t,\xi+t)\ \mathrm{d}(\xi,t) = 0\ ,
\end{align*}
thanks to \eqref{i5}. A similar argument applies to \eqref{eq_i2weak}. 
\end{proof}

%%%%%%%%%%%%%%%%%%%%%%%%%%%%%%%%%%%%%%%%%%%%%%%%%%%%%%%%%%%%%%%%%%%%
%%%%%%%%%%%%%%%%%%%%%%%%%%%%%%%%%%%%%%%%%%%%%%%%%%%%%%%%%%%%%%%%%%%%
\section{Bounded traveling waves without gravity} \label{sec_legere}
%%%%%%%%%%%%%%%%%%%%%%%%%%%%%%%%%%%%%%%%%%%%%%%%%%%%%%%%%%%%%%%%%%%%
%%%%%%%%%%%%%%%%%%%%%%%%%%%%%%%%%%%%%%%%%%%%%%%%%%%%%%%%%%%%%%%%%%%%

Observe that if gravitational effects are neglected, that is, if $G=0$, then \eqref{i5} is an algebraic equation for $H$. We may thus solve this equation for $H$ in terms of $\Gamma$ and so reduce the problem to a single ordinary differential equation for $\Gamma$.

%%%%%%%%%%%%%%%%%%%%%%%%%%%%%%%%%%%%%%%%%%%%%%%%%%%%%%%%%%%%%%%%%%%%
%%%%%%%%%%%%%%%%%%%%%%%%%%%%%%%%%%%%%%%%%%%%%%%%%%%%%%%%%%%%%%%%%%%%
\subsection{No surface diffusion $D=0$}\label{S3.1}
%%%%%%%%%%%%%%%%%%%%%%%%%%%%%%%%%%%%%%%%%%%%%%%%%%%%%%%%%%%%%%%%%%%%
%%%%%%%%%%%%%%%%%%%%%%%%%%%%%%%%%%%%%%%%%%%%%%%%%%%%%%%%%%%%%%%%%%%%

Neglecting diffusion as well reduces \eqref{i5}-\eqref{i6} to 
\begin{align}
2H - H^2 \sigma'(\Gamma) \Gamma' & =  2 H_*\ , & \xi\in \RR\ , \label{a1}\\
\Gamma - H \Gamma \sigma'(\Gamma) \Gamma' & =  0 \ , & \xi\in \RR\ . \label{a2}
\end{align}
Multiplying \eqref{a1} by $\Gamma$ and \eqref{a2} by $-H$ and adding the result we obtain
$$
\Gamma \left( H - 2 H_* \right) = 0\ ,
$$
whence
\begin{equation*}
H(\xi) = 2 H_* \ , \quad \xi \in \mathcal{P}_\Gamma := \left\{ \zeta\in\RR\ : \ \Gamma(\zeta) > 0 \right\}\ . \label{a3}
\end{equation*}
Observe that \eqref{i6i} requires $H_*\ge 0$. Next, \eqref{a2} gives
$$
2 H_* \sigma'(\Gamma) \Gamma' = 1 \;\;\text{ in }\;\; \mathcal{P}_\Gamma\ .
$$
Since $\sigma'<0$ by \eqref{i3}, it readily follows from the previous identity that $H_*>0$ and $\Gamma$ is decreasing on $\mathcal{P}_\Gamma$. Therefore, $\mathcal{P}_\Gamma$ is an interval of the form $(-\infty,\xi_0)$ for some $\xi_0\in (-\infty,\infty]$ and 
\begin{equation}\label{plus}
\sigma(\Gamma(\xi)) = K_3 + \frac{\xi}{2H_*}\ , \qquad \xi\in (-\infty,\xi_0)\ ,
\end{equation}
for some constant $K_3\in\RR$. As $\sigma$ maps $[0,1)$ onto $(-\infty,\sigma_0]$ by  \eqref{i3}, we deduce that $\xi_0 = 2 H_* (\sigma_0-K_3)$ and thus $\xi_0<\infty$. Owing to the translation invariance of \eqref{i5}-\eqref{i6}, we may then choose $\xi_0=0$ so that $\mathcal{P}_\Gamma = (-\infty,0)$ and
\begin{equation}
\Gamma(\xi) = \left\{
\begin{array}{cl}
\displaystyle{\sigma^{-1}\left( \sigma_0 + \frac{\xi}{2H_*}  \right)} \ ,& \xi\in (-\infty,0]\ , \\
 & \\
 0 \ , & \xi\in [0,\infty)\ .
\end{array}
\right. \label{a4}
\end{equation}
Clearly, $\Gamma$ is non-increasing on $\R$. Since $H=H_*$ in $\RR\setminus\mathcal{P}_\Gamma =[0,\infty)$ by \eqref{a1} and \eqref{a4}, we conclude 
\begin{equation}
H(\xi) = \left\{
\begin{array}{cl}
2H_* \ , & \xi\in (-\infty,0]\ , \\
 & \\
 H_* \ , & \xi\in [0,\infty)\ .
\end{array}
\right. \label{a5}
\end{equation}
 We finally deduce from Lemma~\ref{TWws} that we obtain from \eqref{a4}-\eqref{a5} a weak solution $(h,\gamma)$ to \eqref{i1}-\eqref{i2} on $\R$ by \eqref{i4}. This proves Theorem~\ref{T1} and Proposition~\ref{P222} when $G=0$ and $D=0$.

%%%%%%%%%%%%%%%%%%%%%%%%%%%%%%%%%%%%%%%%%%%%%%%%%%%%%%%%%%%%%%%%%%%%
%%%%%%%%%%%%%%%%%%%%%%%%%%%%%%%%%%%%%%%%%%%%%%%%%%%%%%%%%%%%%%%%%%%%
\subsection{Including surface diffusion $D>0$}
%%%%%%%%%%%%%%%%%%%%%%%%%%%%%%%%%%%%%%%%%%%%%%%%%%%%%%%%%%%%%%%%%%%%
%%%%%%%%%%%%%%%%%%%%%%%%%%%%%%%%%%%%%%%%%%%%%%%%%%%%%%%%%%%%%%%%%%%%

Still neglecting gravity but now including surface diffusion, $(H,\Gamma)$ solves
\begin{align}
2 H - H^2 \sigma'(\Gamma) \Gamma' & =  2 H_*\ , & \xi\in \RR\ , \label{b1} \\
\Gamma - \left[ \Gamma \sigma'(\Gamma) H - D \right] \Gamma' & =  0 \ , & \xi\in \RR\ . \label{b2}
\end{align}
Multiplying \eqref{b1} by $(\Gamma \sigma'(\Gamma) H - D)$ and \eqref{b2} by $-\sigma'(\Gamma) H^2$ and adding the resulting identities give a quadratic equation satisfied by $H$ which reads
\begin{equation}
\Gamma \sigma'(\Gamma) H^2 - 2 \left( D + \Gamma \sigma'(\Gamma) H_* \right) H + 2 D H_* = 0\ . \label{b3}
\end{equation}
The equation \eqref{b3} reduces to the identity $H=H_*$ if $\Gamma=0$ while it has only one positive zero whatever the value of $\Gamma>0$ which is given by
$$
H = \frac{D + \Gamma \sigma'(\Gamma) H_* - \sqrt{D^2 + (\Gamma \sigma'(\Gamma) H_*)^2}}{\Gamma \sigma'(\Gamma)}
$$
(recall that $\sigma'<0$ by \eqref{i3}). Consequently, introducing again the positivity set $\mathcal{P}_\Gamma$ defined by 
$$
\mathcal{P}_\Gamma := \left\{ \xi\in\RR\ : \ \Gamma(\xi) > 0 \right\}\ ,
$$
we conclude that
\begin{subequations}
\begin{equation}
H(\xi) = \frac{D + \Gamma \sigma'(\Gamma) H_* - \sqrt{D^2 + (\Gamma \sigma'(\Gamma) H_*)^2}}{\Gamma \sigma'(\Gamma)}(\xi) \ , \quad \xi\in \mathcal{P}_\Gamma\ , \label{b4a}
\end{equation}
\begin{equation}
 H(\xi) = H_* \ , \quad \xi\not\in\mathcal{P}_\Gamma\ . \label{b4b}
\end{equation}
\label{b4}
\end{subequations}
We now infer from \eqref{b2} and \eqref{b4a} that
$$
\left[ \Gamma \sigma'(\Gamma) H_* - \sqrt{D^2 + (\Gamma \sigma'(\Gamma) H_*)^2} \right] \frac{\Gamma'}{\Gamma} = 1 \ , \qquad \xi\in\mathcal{P}_\Gamma\ .
$$
Since $\sigma'<0$ by \eqref{i3} and since \eqref{i5}-\eqref{i6} is invariant with respect to translations, we readily deduce that $\Gamma$ is decreasing on $\mathcal{P}_\Gamma$ and that there is $\xi_0\in (0,\infty]$ such that $\mathcal{P}_\Gamma = (-\infty,\xi_0)$. Integrating the previous differential equation, we end up with
\begin{equation}
\int_{\Gamma(0)}^{\Gamma(\xi)} \left[ z \sigma'(z) H_* - \sqrt{D^2 + (z \sigma'(z) H_*)^2} \right] \frac{\mathrm{d}z}{z} = \xi\ , \qquad \xi\in (-\infty,\xi_0)\ . \label{210}
\end{equation}
Moreover, owing to \eqref{i3} we have
$$
\frac{z \sigma'(z) H_* - \sqrt{D^2 + (z \sigma'(z) H_*)^2}}{z} \mathop{\sim}_{z\to 1} 2 H_* \sigma'(z) \;\;\text{ and }\;\; \frac{z \sigma'(z) H_* - \sqrt{D^2 + (z \sigma'(z) H_*)^2}}{z} \mathop{\sim}_{z\to 0} - \frac{D}{z}\ ,
$$
from which we conclude that necessarily $\xi_0=\infty$, that is, $\mathcal{P}_\Gamma=\RR$, and
\begin{equation}
\lim_{\xi\to -\infty} \Gamma(\xi) = 1\ , \qquad \lim_{\xi\to \infty} \Gamma(\xi) = 0\ . \label{b6}
\end{equation}
Consequently, we realize that $H$ is given by \eqref{b4a}  for $\xi \in \mathbb R$ and satisfies 
$$
H_* < H(\xi) < 2 H_* \ ,\qquad \xi\in \mathbb{R}\ ,
$$
with
$$
\lim_{\xi\to -\infty} H(\xi) = 2H_*\ , \qquad \lim_{\xi\to \infty} H(\xi) = H_*\ ,
$$
thanks to \eqref{i3} and \eqref{b6}. Observe that $\Gamma$ is $C^1$-smooth by applying \eqref{210} (recall that $\sigma \in C^1([0,1))$ and $\Gamma(\RR) \subset (0,1)$) and that $H$ is thus  continuous by \eqref{b4a}. Therefore, we have shown Theorem~\ref{T1} and Proposition~\ref{P222} when $G=0$ and $D>0$.

%%%%%%%%%%%%%%%%%%%%%%%%%%%%%%%%%%%%%%%%%%%%%%%%%%%%%%%%%%%%%%%%%%%%
%%%%%%%%%%%%%%%%%%%%%%%%%%%%%%%%%%%%%%%%%%%%%%%%%%%%%%%%%%%%%%%%%%%%
\section{Bounded traveling waves with gravity} \label{sec_grave}
%%%%%%%%%%%%%%%%%%%%%%%%%%%%%%%%%%%%%%%%%%%%%%%%%%%%%%%%%%%%%%%%%%%%
%%%%%%%%%%%%%%%%%%%%%%%%%%%%%%%%%%%%%%%%%%%%%%%%%%%%%%%%%%%%%%%%%%%%

We shall now include gravity by taking $G>0$ in \eqref{i5} in which case this is also an ordinary differential equation. Thus, we have to use a different approach than before and  assume now that \eqref{assumrho} holds.

%%%%%%%%%%%%%%%%%%%%%%%%%%%%%%%%%%%%%%%%%%%%%%%%%%%%%%%%%%%%%%%%%%%%
%%%%%%%%%%%%%%%%%%%%%%%%%%%%%%%%%%%%%%%%%%%%%%%%%%%%%%%%%%%%%%%%%%%%
\subsection{No surface diffusion $D=0$}
%%%%%%%%%%%%%%%%%%%%%%%%%%%%%%%%%%%%%%%%%%%%%%%%%%%%%%%%%%%%%%%%%%%%
%%%%%%%%%%%%%%%%%%%%%%%%%%%%%%%%%%%%%%%%%%%%%%%%%%%%%%%%%%%%%%%%%%%%

We combine \eqref{i5}-\eqref{i6} to obtain a system of differential equations for $(H,\Gamma)$. More precisely, recalling that
\begin{align}
H + \frac{G H^3}{3} H' - \frac{H^2}{2} \sigma'(\Gamma) \Gamma' & =  H_*\ , & \xi\in \RR\ , \label{c1} \\
\Gamma + \frac{G}{2} \Gamma  H^2 H' - \Gamma \sigma'(\Gamma) H \Gamma' & =  0 \ , & \xi\in \RR\ , \label{c2}
\end{align}
we multiply \eqref{c1} by $12\Gamma$ and \eqref{c2} by $-6H$ and add the resulting identities to obtain 
\begin{equation}
\Gamma \left[ G H^3 H' + 6 \left( H- 2 H_* \right) \right] = 0\ , \qquad \xi\in\RR \label{c3}\ .
\end{equation}
We next multiply \eqref{c1} by $6\Gamma$ and \eqref{c2} by $-4H$ and find, after adding the resulting identities,
\begin{equation}
\Gamma \left[ \sigma'(\Gamma) H^2 \Gamma' + 2 \left( H - 3H_* \right) \right] = 0\ , \qquad \xi\in\RR\ . \label{c4}
\end{equation}
Observe that \eqref{c3} and \eqref{c4} are obviously satisfied in the zero set of $\Gamma$.
In order to analyze \eqref{c3}-\eqref{c4}, we first investigate the behavior of the following auxiliary system of ordinary differential equations:
\begin{eqnarray}
H^3 H' & = & \frac{6}{G} \left( 2 H_* - H \right)\ , \qquad \xi\in \mathbb{R}\ , \label{c5}\\
H^2 \sigma'(\Gamma) \Gamma' & = & 2 \left( 3H_* - H \right)\ , \qquad \xi\in \mathbb{R}\ . \label{c6}
\end{eqnarray}
Recalling \eqref{assumrho}, non-extendable solutions to this system of ordinary differential equations for $(H,\Gamma)$ in $(0,\infty) \times (0,1)$ are provided by the Cauchy-Lipschitz theorem.

%%%%%%%%%%%%%%%%%%%%%%%%%%%%%%%%%%%%%%%%%%%%%%%%%%%%%%%%%%%%%%%%%%%%
\begin{lemma}\label{PL}
Given $(H_0,\Gamma_0)\in (0,\infty)\times (0,1)$, there are unique $\xi_\omega>0$, $\xi_\alpha>0$, and a unique non-extendable solution 
$$
(H,\Gamma) \in C^1((-\xi_\alpha,\xi_\omega),(0,\infty)\times (0,1))
$$
to \eqref{c5}-\eqref{c6} with $(H,\Gamma)(0)=(H_0,\Gamma_0)$. 
At a possibly finite boundary point of $(-\xi_\alpha,\xi_\omega)$, the function~$H$ approaches zero or becomes unbounded, or the function $\Gamma$ approaches zero or $1$. 
\end{lemma}
%%%%%%%%%%%%%%%%%%%%%%%%%%%%%%%%%%%%%%%%%%%%%%%%%%%%%%%%%%%%%%%%%%%%

As we now shall see, the behavior of $(H,\Gamma)$ varies according to the value of $H_0$.

%%%%%%%%%%%%%%%%%%%%%%%%%%%%%%%%%%%%%%%%%%%%%%%%%%%%%%%%%%%%%%%%%%%%
\begin{lemma}\label{L22}
\noindent (i) If $H_0=2H_*$, then
\begin{equation}\label{uu}
\xi_\alpha = \infty \ ,\qquad \xi_\omega = 2 H_* \left( \sigma_0 - \sigma(\Gamma_0) \right) < \infty\ . 
\end{equation}
Moreover, $H\equiv 2H_*$ on $(-\infty,\xi_\omega]$ while $\Gamma$ is a decreasing function from $(-\infty,\xi_\omega]$ onto $[0,1)$.

\noindent (ii) If $H_0\in (0,2H_*)$, then both $\xi_\alpha$ and $\xi_\omega$ are finite with $H(-\xi_\alpha)=0$ and $\Gamma(\xi_\omega)=0$. Moreover,~$H$ is increasing and $\Gamma$ is decreasing on $(-\xi_\alpha,\xi_\omega)$.

\noindent (iii) If $H_0\in (2H_*,\infty)$, then both $\xi_\alpha$ and $\xi_\omega$ are finite with $\Gamma(-\xi_\alpha)=\Gamma(\xi_\omega)=0$. Moreover, $H$ is decreasing on $(-\xi_\alpha,\xi_\omega)$ and stays above $2H_*$.
\end{lemma}
%%%%%%%%%%%%%%%%%%%%%%%%%%%%%%%%%%%%%%%%%%%%%%%%%%%%%%%%%%%%%%%%%%%%

\begin{proof}
We first note that \eqref{c5} is independent of $\Gamma$ and has an explicit constant solution $H\equiv 2H_*$. Thus, we distinguish three cases relating the initial value $H_0$ to this particular value.

\medskip

\noindent (i) Let $H_0=2H_*$. Clearly, $H\equiv 2H_*$ in $(-\xi_\alpha,\xi_\omega)$ by \eqref{c5}, and we deduce from \eqref{c6} that $\Gamma$ is decreasing with
\begin{equation*}
\sigma(\Gamma(\xi)) = \sigma(\Gamma_0) + \frac{\xi}{2H_*} \ , \qquad \xi\in (-\xi_\alpha,\xi_\omega)\ . %\label{c12}
\end{equation*}
The properties \eqref{i3} of $\sigma$ then entail \eqref{uu}.

\medskip

\noindent (ii) Let $H_0\in (0,2H_*)$. We first observe that \eqref{c5} guarantees that $H$ stays below $2H_*$ so that $H^4$ increases on $(-\xi_\alpha,\xi_\omega)$. Thus $0<H(\xi)<H_0<2H_*$ for $\xi\in (-\xi_\alpha,0)$. It follows from \eqref{c5}, \eqref{c6}, and this upper bound that, for $\xi\in (-\xi_\alpha, 0)$,
$$
(H^4)'\ge \frac{24 (2H_*-H_0)}{G}\ ,\qquad (\sigma(\Gamma))'\ge \frac{2H_*}{H^2} \ge \frac{1}{2 H_*} >0\ ,
$$
from which we deduce that $\Gamma$ is decreasing on $(-\xi_\alpha, 0)$ and
$$
H_0^4 \ge H_0^4 - \lim_{\xi\to -\xi_\alpha} H^4(\xi) \ge \frac{24}{G} \left( 2H_* - H_0 \right) \xi_\alpha\ .
$$
Therefore, $\xi_\alpha$ is finite. Owing to the monotonicity and positivity of $H$, the limit $\ell:=\lim_{\xi\to -\xi_\alpha} H(\xi)$ is a well-defined non-negative real number. Our aim is to show that $\ell=0$. By contradiction assume that $\ell>0$. Then, the definition of $\xi_\alpha$ and the just proven monotonicity of $\Gamma$ imply that $\Gamma(\xi)\to 1$ as $\xi\to -\xi_\alpha$. Consequently, $\sigma(\Gamma(\xi))\to -\infty$ as $\xi\to -\xi_\alpha$. Now, the monotonicity of $H$ guarantees $0<\ell < H(\xi)<H_0$ for $\xi\in (-\xi_\alpha, 0)$ and we infer from \eqref{c6} that
$$
\sigma(\Gamma)'(\xi)\le \frac{6 H_*}{\ell^2}\ .
$$
Integrating this inequality yields
$$
\sigma(\Gamma_0)-\frac{6H_*\xi_\alpha}{\ell^2}\le \sigma(\Gamma_0){+ \frac{6H_*\xi}{\ell^2}}\le \sigma(\Gamma(\xi))
$$
for $\xi\in (-\xi_\alpha,0)$. The above bound  clearly yields a contradiction as it prevents $\sigma(\Gamma(\xi))$ from reaching $-\infty$ as $\xi\to -\xi_\alpha$.
Thus, we have shown $\ell=H(-\xi_\alpha)=0$ and 
\begin{equation*}
H(-\xi_\alpha)=0 < H(\xi) < H_0\ , \qquad \xi\in (-\xi_\alpha,0)\ . %\label{c9}
\end{equation*}
Let us now study the behavior for positive $\xi$. It readily follows from \eqref{c5} that
\begin{equation}
H_0 < H(\xi) < 2 H_*\ , \qquad \xi\in (0,\xi_\omega)\ . \label{c7}
\end{equation}
We then infer from the negativity \eqref{i3} of $\sigma'$, \eqref{c6}, and \eqref{c7} that $\Gamma$ is decreasing in $(0,\xi_\omega)$ and satisfies
$$
\|\sigma'\|_{L_\infty(0,\Gamma_0)}\, \Gamma'(\xi) \le - \frac{1}{2 H_*}\,, \qquad \xi\in (0,\xi_\omega)\ .
$$
As $\Gamma>0$ in $(0,\xi_\omega)$, the previous inequality readily implies that $\xi_\omega<\infty$. Consequently, $\Gamma(\xi_\omega)=0$ by \eqref{c7},
and thus
\begin{equation*}
\xi_\omega < \infty \;\;\text{ and }\;\; \Gamma(\xi_\omega)=0 < \Gamma(\xi) < \Gamma_0\ , \quad \xi\in (0,\xi_\omega)\ . %\label{c8}
\end{equation*} 
Hence, we obtain (ii).

\medskip

\noindent (iii) Let $H_0>2H_*$. In this case, the differential equation  \eqref{c5} guarantees
\begin{equation}
2 H_* < H(\xi) \;\;\;\text{ and }\;\;\; H'(\xi) < 0 \ , \quad \xi\in (-\xi_\alpha,\xi_\omega)\ . \label{c10}
\end{equation}
Also, integrating \eqref{c6} gives
\begin{equation}
\sigma(\Gamma(\xi)) = \sigma(\Gamma_0) + 2 \int_0^\xi \frac{3H_*-H(\zeta)}{H(\zeta)^2}\ \mathrm{d}\zeta\ , \quad \xi\in 
(-\xi_\alpha,\xi_\omega)\ . \label{c101}
\end{equation}
Assume now for contradiction that $H>3H_*$ on $(-\xi_\alpha,\xi_\omega)$. In particular, $H_0>3H_*$ and this implies that $\sigma(\Gamma)$ is decreasing on $(-\xi_\alpha,\xi_\omega)$ according to \eqref{c6}. Then, by \eqref{c10} and \eqref{c101}, for $\xi\in (0,\xi_\omega)$,
$$
2H_*<H(\xi)< H_0
$$
and $$
\sigma(\Gamma_0)\ge \sigma(\Gamma(\xi))\ge \sigma(\Gamma_0)-2(H_0-3H_*)\int_0^\xi\frac{\rd \zeta}{H(\zeta)^2}\ge 
\sigma(\Gamma_0)-\frac{H_0-3H_*}{2H_*^2}\xi\ .
$$
Since $\sigma$ maps $[0,1)$ onto $(-\infty,\sigma_0]$, we deduce from the previous bounds that $\xi_\omega=\infty$.
However, it is clear from \eqref{c5} that then $H(\xi)\to 2H_*$ as $\xi\to\infty$, whence a contradiction. 
Consequently, there is $\xi_*\in (-\xi_\alpha,\xi_\omega)$ such that $H(\xi_*)=3H_*$. Invoking the monotonicity \eqref{c10} of $H$, we have
\begin{equation}
H(\xi)>3 H_*\ , \quad \xi\in (-\xi_\alpha,\xi_*)\ ,\qquad 
H(\xi)<3H_*\ , \quad \xi\in (\xi_*,\xi_\omega)  \ .
\label{z1}
\end{equation}
 We next deduce from \eqref{c101} that
\begin{equation}
\sigma(\Gamma(\xi)) = \sigma(\Gamma(\xi_*)) + 2 \int_{\xi_*}^\xi \frac{3H_*-H(\zeta)}{H(\zeta)^2}\, \mathrm{d}\zeta\ , \quad \xi\in (-\xi_\alpha,\xi_\omega)\ . \label{c101a}
\end{equation}
On the one hand, the properties \eqref{z1} of $H$ and \eqref{c101a} ensure that $\sigma(\Gamma(\xi))\ge \sigma(\Gamma(\xi_*))$  so that $\Gamma(\xi)\le \Gamma(\xi_*)<1$ for $\xi\in (-\xi_\alpha,\xi_\omega)$. Hence, taking $\xi_1\in (\xi_*,\xi_\omega)$, it follows from \eqref{z1}, \eqref{c101a}, and the properties of $\sigma$ that, for $\xi\in (\xi_1,\xi_\omega)$, 
\begin{align*}
\sigma_0 \ge \sigma(\Gamma(\xi)) \ge & \sigma(\Gamma(\xi_*)) + 2 \int_{\xi_1}^\xi \frac{3H_*-H(\zeta)}{H(\zeta)^2}\,\rd\zeta \ge \sigma(\Gamma(\xi_*)) + 2 \int_{\xi_1}^\xi \frac{3H_*-H(\xi_1)}{H(\zeta)^2}\,\rd\zeta \\ 
\ge & \sigma(\Gamma(\xi_*)) + \frac{2}{9}\ \frac{3H_*-H(\xi_1)}{H_*^2}\ (\xi-\xi_1)\ ,
\end{align*}
from which we readily deduce that $\xi_\omega$ is finite. 
We next show that $\xi_\alpha$ is finite. Observe that \eqref{c5} implies 
$$
H^3H'\ge -\frac{6H}{G}\ ,\quad \xi\in (-\xi_\alpha,\xi_\omega)\ ,
$$
so that, since $H$ does not vanish according to Lemma~\ref{PL},
\begin{equation}
H(\xi)^3\le H_0^3+\frac{18}{G}\vert\xi\vert\ ,\quad \xi\in (-\xi_\alpha,0)\ . \label{volvic}
\end{equation}
Now, taking $\xi_2\in (-\xi_\alpha,\min\{0,\xi_*\})$, we infer from \eqref{z1}, \eqref{c101a}, the properties of $\sigma$, and the previous estimate that, for $\xi\in (-\xi_\alpha,\xi_2)$,
\begin{equation*}
\begin{split}
\sigma_0\ge \sigma(\Gamma(\xi)) & \ge \sigma(\Gamma(\xi_*)) + 2 \int_{\xi}^{\xi_2}\frac{H(\zeta)-3 H_*}{H(\zeta)^2}\,\rd \zeta\\
& \ge \sigma(\Gamma(\xi_*)) + 2 G^{2/3} \big(H(\xi_2)-3H_*\big) \int_\xi^{\xi_2} \frac{\rd \zeta}{(G H_0^3 + 18  \vert\zeta\vert)^{2/3}}\ ,
\end{split}
\end{equation*}
from which the finiteness of $\xi_\alpha$ readily follows as the right-hand side of the above inequality diverges as $\xi\to -\infty$. Consequently, since $H$ cannot reach zero by \eqref{c10} and cannot blow up on a finite interval by \eqref{volvic} while $\Gamma$ stays below 1 on $(-\xi_\alpha,\xi_\omega)$ as shown above, we deduce that $\Gamma(\xi)\to 0$ as $\xi\to -\xi_\alpha$ and $\xi\to \xi_\omega$.
\end{proof}

\medskip

Part~(i) of Lemma~\ref{L22} allows us to construct a solution to \eqref{i5}-\eqref{i6} by extending $\Gamma$ by $\Gamma(\xi)=0$ for $\xi\in [\xi_\omega,\infty)$ and using \eqref{c1} to see that $H$ is determined on the interval $(\xi_\omega,\infty)$ as the solution $\hat H$ to
\begin{equation}
\hat H' = \frac{3 \big( H_* - \hat H \big)}{G \hat H^3 }\ , \quad \xi\in (\xi_\omega,\infty)\ , \qquad \hat H(\xi_\omega) = 2 H_*\ . \label{c13}
\end{equation} 
Therefore, after a suitable translation so that $\xi_\omega=0$, a solution $(H,\Gamma)$ to \eqref{i5}-\eqref{i6} in the case $D=0$ and $G>0$ is given by
\begin{equation}
(H,\Gamma)(\xi) =
\left\{
\begin{array}{rcl}
\left(2H_*,\sigma^{-1}\left(\sigma_0 + \dfrac{\xi}{2H^*}\right) \right)\,, &  \, \xi <0\,,\\[16pt]
\left(\hat H(\xi), 0 \right)\,, &  \, \xi \geq0\,.
\end{array}
\right. \label{c13a}
\end{equation}
Note that a particular consequence of \eqref{c13} and \eqref{c13a} is that $H(\xi)\to H_*$  as $\xi\to \infty$. Finally, $(H,\Gamma)$ has the expected regularity and the pair $(h,\gamma)$ obtained subsequently from $(H,\Gamma)$ through \eqref{i4} is a weak solution to \eqref{i1}-\eqref{i2} on $\R$ by Lemma~\ref{TWws}. This yields Theorem~\ref{T1} and Proposition~\ref{P222} when $G>0$ and $D=0$.

%%%%%%%%%%%%%%%%%%%%%%%%%%%%%%%%%%%%%%%%%%%%%%%%%%%%%%%%%%%%%%%%%%%%
\begin{remark}\label{RR1}
{\rm Since system \eqref{c5}-\eqref{c6} is autonomous, it is a priori not excluded to concatenate different types of solutions constructed in Lemma~\ref{L22} and solutions to \eqref{c13}.  However, due to the fact that traveling wave solutions in the case $G>0,D=0,$ have to be continuous on $\mathbb{R}$, it follows from Lemma~\ref{L22} that solutions with $H_0\ne 2H_*$ cannot be concatenated to a physically relevant traveling wave.}
\end{remark}
%%%%%%%%%%%%%%%%%%%%%%%%%%%%%%%%%%%%%%%%%%%%%%%%%%%%%%%%%%%%%%%%%%%%

%%%%%%%%%%%%%%%%%%%%%%%%%%%%%%%%%%%%%%%%%%%%%%%%%%%%%%%%%%%%%%%%%%%%
%%%%%%%%%%%%%%%%%%%%%%%%%%%%%%%%%%%%%%%%%%%%%%%%%%%%%%%%%%%%%%%%%%%%
\subsection{With surface diffusion $D>0$} \label{sec_32}
%%%%%%%%%%%%%%%%%%%%%%%%%%%%%%%%%%%%%%%%%%%%%%%%%%%%%%%%%%%%%%%%%%%%
%%%%%%%%%%%%%%%%%%%%%%%%%%%%%%%%%%%%%%%%%%%%%%%%%%%%%%%%%%%%%%%%%%%%

We now take both gravity and surface diffusion into account and again transform \eqref{i5}-\eqref{i6} into a system of differential equations for $(H,\Gamma)$. For that purpose, we multiply \eqref{i5} by $12(\Gamma \sigma'(\Gamma) H - D)$ and \eqref{i6} by {$-6 \sigma'(\Gamma) H^2$} and add the resulting identities to obtain
\begin{equation*}
G H^3 \left( \Gamma \sigma'(\Gamma) H - 4 D \right) H' = 6 \Gamma \sigma'(\Gamma) H \left( 2 H_*-H \right) + 12 D \left( H - H_* \right)\ , \qquad \xi\in\RR\ . 
\end{equation*}
Similarly, we multiply \eqref{i5} by $6\Gamma$ and \eqref{i6} by $-4H$ and find, after adding the resulting identities,
\begin{equation*}
H \left( \Gamma \sigma'(\Gamma) H - 4 D \right) \Gamma' = 2 \Gamma \left( 3H_* - H \right)\ , \qquad \xi\in\RR\ . 
\end{equation*}
Thus, $(H,\Gamma)$ solves the following differential system:
\begin{align*}
G H^3 \big( \Gamma \sigma'(\Gamma) H - 4 D \big) H' & =  6 \Gamma \sigma'(\Gamma) H \left( 2 H_*-H \right) + 12  D \left( H - H_* \right)\ , & \xi\in\RR\ , \\ %\label{d1} \\
H \big( \Gamma \sigma'(\Gamma) H - 4 D \big) \Gamma' & = 2 \Gamma \left( 3H_* - H \right)\ , & \xi\in\RR\ . %\label{d2}
\end{align*}
Since $\sigma'<0$ by \eqref{i3}, we note that 
\begin{equation}
\Gamma \sigma'(\Gamma) H - 4 D<0\ , \qquad \xi\in\RR\ .  \label{d3}
\end{equation}
Thus, we shall construct a traveling wave solution for (\ref{i1})-(\ref{i2}) by proving the existence of a heteroclinic orbit for the system (recall that $\varrho=-1/\sigma'$ and $\varrho(1)=0$)
\begin{align}
 H' & = \frac{ 6 \Gamma H \left( 2 H_*-H \right) - 12 D \left( H - H_* \right)\varrho(\Gamma)}{{G H^3} \big( \Gamma  H + 4 D \varrho(\Gamma) \big)}\ , & \xi\in\RR\ ,  \label{app_d1} \\
 \Gamma' & =  \frac{2 \Gamma \left( H-3H_*\right)\varrho(\Gamma)}{H \big( \Gamma  H + 4 D \varrho(\Gamma) \big)}\ , & \xi\in\RR\ , \label{app_d2}
\end{align}
which connects the critical points $(2H_*,1)$ and $(H_*,0)$. In order to do so, set
\begin{align}
f_1(\tilde{H},\tilde{\Gamma}) & := \frac{ 6 \tilde{\Gamma} \tilde{H} ( 2 H_*-\tilde{H} ) - 12 D ( \tilde{H} - H_* )\varrho(\tilde\Gamma)}{{G \tilde{H}^3} ( \tilde{\Gamma}  \tilde{H} + 4 D \varrho(\tilde{\Gamma}) )}, \label{f_d1} \\
 f_2(\tilde{H},\tilde{\Gamma}) & :=  \frac{2 \tilde{\Gamma} ( \tilde{H}-3H_*)\varrho(\tilde{\Gamma})}{\tilde{H} 
(\tilde{\Gamma}  \tilde{H} + 4 D \varrho(\tilde{\Gamma}) )}, \label{f_d2}
\end{align}
and note that, by \eqref{assumrho},
$$
(f_1,f_2)\in \mathrm{Lip}((0,\infty)\times (0,1) , \mathbb{R}^2).
$$
Thus, given $(H_0,\Gamma_0)\in (0,\infty)\times (0,1)$, the classical Cauchy-Lipschitz theorem ensures the existence of a unique non-extendable  solution 
$$
(H,\Gamma)\in C^1((-\xi_\alpha,\xi_\omega) , (0,\infty)\times (0,1))
$$
of \eqref{app_d1}-\eqref{app_d2} complemented with the initial condition
\begin{equation}\label{app_d3}
(H,\Gamma)(0) = (H_0,\Gamma_0)\ ,
\end{equation}
and global existence being characterized as in Lemma~\ref{PL}.

%%%%%%%%%%%%%%%%%%%%%%%%%%%%%%%%%%%%%%%%%%%%%%%%%%%%%%%%%%%%%%%%%%%%
\begin{proposition}[Blow-up scenario] \label{prop_app1}
Given $(H_0,\Gamma_0) \in [H_*,2H_*] \times (0,1)$, let 
$$
(H,\Gamma)\in C^1((-\xi_\alpha,\xi_\omega) , \mathbb{R}^2)
$$
be the non-extendable solution to \eqref{app_d1}-\eqref{app_d3}. Then the following assertions hold true:

\begin{itemize}

\item [(a)] $\xi_{\omega} = \infty$ and 
\begin{eqnarray*}
(H(\xi),\Gamma(\xi)) \in (H_*,2H_*) \times (0,1)\ , \quad \xi \in (0,\infty)\,.
\end{eqnarray*}
Moreover, $\Gamma$ is decreasing on $(-\xi_\alpha,\infty)$.

\item [(b)] If $\xi_{\alpha}$ is finite, then there exists $\tilde{\xi}_{\alpha} < \xi_{\alpha}$ such that $H(-\tilde{\xi}_{\alpha}) \in \{H_*,2H_*\}.$

\end{itemize}
\end{proposition}
%%%%%%%%%%%%%%%%%%%%%%%%%%%%%%%%%%%%%%%%%%%%%%%%%%%%%%%%%%%%%%%%%%%%

\begin{proof}
(a) Let $(H_0,\Gamma_0)$ and $(H,\Gamma)$ be as above. Since
\begin{equation} \label{eq_f1}
f_1(H_*,\tilde{\Gamma}) > 0\,, \quad f_1(2H_*,\tilde{\Gamma}) < 0 \ , \quad\tilde{\Gamma} \in (0,1)\,,
\end{equation}
we conclude 
$$
H(\xi) \in (H_*,2H_*) \ , \quad\xi \in (0,\xi_{\omega})\ .
$$ 
We also have
$$
f_2(\tilde{H},\tilde{\Gamma}) < 0 \ , \quad (\tilde{H},\tilde{\Gamma}) \in [H_*,2H_*] \times (0,1)\ .
$$
Hence, $\Gamma$ is decreasing on $(-\xi_\alpha,\xi_{\omega})$. Finally, since $\varrho\ge 0$ vanishes only at $\tilde{\Gamma}=1$ there exists $\mu>0$ such that
\begin{equation}\label{K_-}
 0< \mu \le \tilde{\Gamma} \tilde{H} + 4D \varrho(\tilde{\Gamma})  \quad \text{for} \quad (\tilde{H},\tilde{\Gamma}) \in (H_*,2H_*) \times (0,1)\,,
\end{equation}
and we conclude that
$$
\Gamma'(\xi) \geq - \frac{4 R_\varrho}{\mu} \Gamma(\xi) \ , \quad\xi \in (0,\xi_{\omega})\,.
$$
Consequently,
$$
\Gamma(\xi) \ge \Gamma_0 \ e^{- 4 R_\varrho \xi/\mu}  \ , \quad\xi \in (0,\xi_{\omega})\ ,
$$
and $\Gamma$ thus cannot reach zero for $\xi$ finite, which rules out the possibility that $\xi_\omega$ is finite. This proves~(a). 

\medskip

\noindent (b) Assume that $\xi_\alpha <\infty.$  Assume also for contradiction that 
\begin{equation}
H_*<H(\xi)<2H_* \ , \quad \xi\in(-\xi_\alpha, 0),\label{luchon}
\end{equation}
and recall that $\Gamma$ is then decreasing on $(-\xi_\alpha,\infty)$ so that
\begin{equation}
\Gamma(\xi)\ge \Gamma_0 \ , \quad \xi\in(-\xi_\alpha, 0)\ . \label{evian}
\end{equation}
On the one hand, owing to \eqref{evian} and \eqref{luchon}, the finiteness of $-\xi_\alpha$ implies that 
$$
\lim_{\xi\to -\xi_\alpha} \Gamma(\xi) = 1\ . 
$$
On the other hand, recalling that $\varrho=-1/\sigma'$, we infer from \eqref{app_d2}, \eqref{K_-}, and \eqref{luchon} that
$$
\dfrac{\rd}{\rd\xi}\sigma(\Gamma)
= - \frac{2 \Gamma \left( H-3H_*\right)}{H \big( \Gamma  H + 4 D \varrho(\Gamma) \big)} \le \frac{4 H_* \Gamma}{H \big( \Gamma  H + 4 D \varrho(\Gamma) \big)} \le \dfrac{4}{\mu}\quad \text{in}\quad (-\xi_{\alpha},0).
$$ 
Integration over $(-\xi_\alpha,0)$ and application of the decreasing function $\sigma^{-1}$ then yield
$$
\Gamma(\xi) \leq \sigma^{-1} \left( \sigma(\Gamma_0) - \dfrac{4\xi_{\alpha} }{\mu} \right) < 1  \ , \quad \xi \in (-\xi_{\alpha},0)\ .
$$ 
Therefore, $\Gamma$ cannot reach the value $1$ as $\xi\to- \xi_\alpha$ and we obtain a contradiction. 
\end{proof}

We next analyze the behavior of $(H,\Gamma)$ as $\xi$ tends to $\infty$.

%%%%%%%%%%%%%%%%%%%%%%%%%%%%%%%%%%%%%%%%%%%%%%%%%%%%%%%%%%%%%%%%%%%%
\begin{proposition}[The $\omega$-limit set] \label{prop_step2} Let $(H_0,\Gamma_0)$ and $(H,\Gamma)$ be as in Proposition~\ref{prop_app1}.
Then 
$$
\lim_{\xi \to  \infty} \big(H(\xi),\Gamma(\xi)\big) = (H_*,0)\ . 
$$
\end{proposition} 
%%%%%%%%%%%%%%%%%%%%%%%%%%%%%%%%%%%%%%%%%%%%%%%%%%%%%%%%%%%%%%%%%%%%

\begin{proof}
Let $(H_0,\Gamma_0)$ and $(H,\Gamma)$ be as in Proposition~\ref{prop_app1}. We recall that $H(\xi) \leq 2H_*$ for all $\xi \in (0,\infty)$ and  that $\Gamma$ is decreasing on $(0,\infty).$ Given $\xi\ge 0$, it follows from these properties and \eqref{app_d2} that
$$
-\Gamma' \le  \frac{2 \Gamma \left( 3H_*-H\right)\varrho(\Gamma)}{4D H \varrho(\Gamma) }\le \frac{\Gamma}{D}
$$
and
$$
-\Gamma' \ge  \frac{2 \Gamma H_*}{ H \big(4D- \Gamma H\sigma'(\Gamma)\big)}\ge \frac{\Gamma }{ 4D +2 \Gamma_0 H_*\|\sigma'\|_{L_\infty(0,\Gamma_0)}}\ .
$$
Consequently,
\begin{equation}
e^{-\xi/D} \le \frac{\Gamma(\xi)}{\Gamma_0} \le  e^{-\delta_0 \xi} \ , \quad \xi\in [0,\infty)\ , \;\;\;\text{ with }\;\;\; \delta_0 := \frac{1}{4D+2\Gamma_0 H_* \|\sigma'\|_{L_\infty(0,\Gamma_0)} }>0\ . \label{vittel}
\end{equation}
We next turn to $H$ and deduce from \eqref{app_d1} and Proposition~\ref{prop_app1}~(a) that, for $ \xi\in (0,\infty)$,
$$
H' + \frac{3 (H-H_*)}{G H_*^3} \ge 0
$$
and
$$
H' + \frac{3 D}{2G H_*^3}\ \frac{\min_{[0,\Gamma_0]}{\{\varrho\}}}{2 H_* + 4 D R_\varrho}\ (H-H_*) \le \frac{3\Gamma}{G H_* D \min_{[0,\Gamma_0]}{\{\varrho\}}}\ .
$$
Combining the previous differential inequalities with \eqref{vittel} 
ensure that there are positive constants $C_1$, $C_2$, $\delta_1$, $\delta_2$ depending on $G$, $H_*$, $D$, $\sigma'$, $H_0$, and $\Gamma_0$ such that
\begin{equation*}
C_1\ e^{-\delta_1 \xi} \le H(\xi)-H_* \le C_2\ e^{-\delta_2 \xi}\,, \qquad \xi\in (0,\infty)\ .
\end{equation*}
Therefore, $(H,\Gamma)(\xi)\to (H_*,0)$ as $\xi\to\infty$.
\end{proof}

\medskip

Our next goal is to show the existence of a global solution to \eqref{app_d1}-\eqref{app_d3}, that is, a solution with $\xi_\alpha=\infty$.  To this end several steps are needed. We begin with the behavior as $\xi\to -\infty$ of global solutions (if any) to check that they have the expected limit. 

%%%%%%%%%%%%%%%%%%%%%%%%%%%%%%%%%%%%%%%%%%%%%%%%%%%%%%%%%%%%%%%%%%%%
\begin{proposition}\label{prglobal}
Assume that $(H,\Gamma)$ is a global solution (that is, $\xi_\omega=\xi_\alpha=\infty$) to \eqref{app_d1}-\eqref{app_d2} such that $(H,\Gamma)(\xi) \in (H_*,2H_*) \times (0,1)$ for all $\xi\in\RR.$
Then
$$
\lim_{\xi \to -\infty} \big(H(\xi),\Gamma(\xi)\big) = (2H_*,1)\ .
$$
\end{proposition}
%%%%%%%%%%%%%%%%%%%%%%%%%%%%%%%%%%%%%%%%%%%%%%%%%%%%%%%%%%%%%%%%%%%%

\begin{proof}
Since $H(\xi) \in (H_*,2H_*)$ for all $\xi\in\mathbb{R}$, it follows from \eqref{rho1}, \eqref{app_d2}, and the non-negativity of $\varrho$ that 
\begin{equation} \label{eq_Gammap}
\Gamma' \le  - \frac{2 \Gamma H_* \varrho(\Gamma)}{H ( \Gamma H + 4 D \varrho(\Gamma))} \le - \frac{\Gamma \varrho(\Gamma)}{2 H_* + 4 D R_\varrho} \le 0\,, \quad \xi \in \RR\,.
\end{equation}
Thus $\Gamma$ is a decreasing and bounded function on $\mathbb{R}$ ranging in $(0,1)$ and there is $\Gamma_\alpha\in [0,1]$ such that $\Gamma(\xi) \to \Gamma_\alpha$ as $\xi\to-\infty.$ Furthermore, integrating \eqref{eq_Gammap} over $(\xi,0)$ with $\xi<0$ and using \eqref{rho1} and the bounds $0<\Gamma<1$ give
$$
\frac{1}{2H_* + 4D R_\varrho} \int_\xi^0 \Gamma(\zeta) \varrho(\Gamma(\zeta))\ \mathrm{d}\zeta \le \Gamma(\xi) \le 1\ , \quad \xi\in (-\infty,0)\ .
$$
Since both $\Gamma$ and $\varrho(\Gamma)$ are non-negative, the above inequality and \eqref{eq_Gammap} entail that 
\begin{equation}
0 \le \Gamma_0 \int_{-\infty}^0 \varrho(\Gamma(\zeta))\ \mathrm{d}\zeta \le \int_{-\infty}^0 \Gamma(\zeta) \varrho(\Gamma(\zeta))\ \mathrm{d}\zeta \le 2H_* + 4D R_\varrho\ . \label{eq_IGP}
\end{equation}
Collecting the above information and using the continuity \eqref{rho1} of $\varrho$ we realize that $\varrho(\Gamma)$ belongs to $L_1(-\infty,0)$ with $\varrho(\Gamma(\xi))\to \varrho(\Gamma_\alpha)$ as $\xi\to -\infty$ and thus $\varrho(\Gamma_\alpha)=0$. Since $\varrho$ only vanishes once in $[0,1]$ by \eqref{rho1} we conclude that $\Gamma_\alpha=1$, that is,  
\begin{equation*} 
\lim_{\xi\to -\infty} \Gamma(\xi) = 1\ . 
\end{equation*}
Next, since $H(\xi) \in (H_*,2H_*)$ for all $\xi\in\mathbb{R}$, we infer from \eqref{rho1} and \eqref{K_-} that, for $\xi<0$, 
\begin{equation}
\delta_3 := \frac{3 \Gamma_0}{4 G H_*^2 (2H_* + 4D R_\varrho)} \le \frac{6 \Gamma H}{G H^3 \big( \Gamma H + 4D \varrho(\Gamma) \big)} \le \frac{12}{\mu G H_*^2} \label{chopin}
\end{equation}
and
\begin{equation}
0 \le \frac{12 D (H-H_*)}{G H^3 \big( \Gamma H + 4D \varrho(\Gamma) \big)} \le \frac{12D}{\mu G H_*^2}\ . \label{beethoven}
\end{equation}
Therefore, by \eqref{app_d1},
$$
0 \le \delta_3 (2H_*-H) \le \frac{6 \Gamma H (2H_*-H)}{G H^3 \big( \Gamma H + 4D \varrho(\Gamma) \big)} = H' +  \frac{12 D (H-H_*) \varrho(\Gamma)}{G H^3 \big( \Gamma H + 4D \varrho(\Gamma) \big)} \le H' + \frac{12D}{\mu G H_*^2} \varrho(\Gamma)
$$
for $\xi<0$. Integrating the above inequality over $(\xi,0)$ with $\xi<0$ gives
$$
0 \le \delta_3 \int_\xi^0 (2H_*-H(\zeta))\ \mathrm{d}\zeta \le H_0 - H(\xi) + \frac{12D}{\mu G H_*^2} \int_\xi^0 \varrho(\Gamma(\zeta))\ \mathrm{d}\zeta\ , 
$$
which, together with \eqref{eq_IGP}, implies that
\begin{equation}
2H_* - H \in L_1(-\infty,0)\ . \label{debussy}
\end{equation}
Moreover, using again \eqref{app_d1}, \eqref{chopin}, and \eqref{beethoven}, we find
$$
|H'| \le \frac{12}{\mu G H_*^2} \left[  (2H_*-H) + D \varrho(\Gamma) \right]\ ,\quad \xi \in (-\infty,0)\ ,
$$
and deduce from \eqref{eq_IGP} and \eqref{debussy} that
\begin{equation}
H' \in L_1(-\infty,0)\ . \label{ravel}
\end{equation}
Combining \eqref{debussy} and \eqref{ravel} entails $H(\xi)\to 2H_*$ as $\xi\to -\infty$ and completes the proof.
\end{proof}

\medskip

The next step is a more precise study of the phase plane associated to \eqref{app_d1}-\eqref{app_d2} which requires a refined analysis of $f_1$ defined in \eqref{f_d1}. Actually, $f_1$ also reads
$$
f_1(\tilde{H},\tilde{\Gamma}) = \frac{12D \tilde{\Gamma} (\tilde{H}-H_*)}{G \tilde{H}^3 ( \tilde{\Gamma}  \tilde{H} + 4 D \varrho(\tilde{\Gamma}) )} \left[  \frac{\tilde{H} (2 H_*-\tilde{H})}{2D (\tilde{H}-H_*)} - g(\tilde{\Gamma}) \right]\ , \quad (\tilde{H},\tilde{\Gamma})\in (H_*,2H_*)\times (0,1)\ ,
$$ 
with 
$$
g(\tilde{\Gamma}) := \frac{\varrho(\tilde{\Gamma})}{\tilde{\Gamma}}\ , \quad \tilde{\Gamma}\in (0,1)\ ,
$$
so that
\begin{equation}
f_1(\tilde{H},\tilde{\Gamma}) = 0 \quad \text{ if and only if } \quad \phi(\tilde{H}) := \frac{\tilde{H} (2 H_*-\tilde{H})}{2D (\tilde{H}-H_*)} = g(\tilde{\Gamma})\ . \label{eq_Gammat}
\end{equation}
In addition, the monotonicity of $\phi$ on $(H_*,2H_*)$ together with \eqref{eq_f1} ensure that, given $\tilde{\Gamma}\in (0,1)$, there is a unique $H_c(\tilde{\Gamma})\in (H_*,2H_*)$ such that
\begin{equation}
\begin{split}
f_1(\tilde{H},\tilde{\Gamma}) < & 0 \quad \text{ for } \quad \tilde{H} \in \big( H_c(\tilde{\Gamma}) , 2H_* \big)\ , \\
f_1(\tilde{H},\tilde{\Gamma}) > & 0 \quad \text{ for } \quad \tilde{H} \in \big( H_*, H_c(\tilde{\Gamma}) \big)\ ,
\end{split} \label{mozart}
\end{equation}
and $H_c\in C((0,1),(H_*,2H_*))$ satisfies $H_c(\tilde{\Gamma})\to 2H_*$ as $\tilde{\Gamma}\to 1$ and $H_c(\tilde{\Gamma})\to H_*$ as $\tilde{\Gamma}\to 0$.
Next, since $\varrho$ is Lipschitz continuous in $[0,1/2]$, $\varrho(0)>0$, and $g'(\tilde{\Gamma})=\big( \tilde{\Gamma} \varrho'(\tilde{\Gamma}) -\varrho(\tilde{\Gamma}) \big)/ \tilde{\Gamma}^2$ for a.a. $\tilde{\Gamma}\in [0,1/2]$ by \eqref{assumrho}, the function $g$ is  decreasing in a right-neighborhood of zero. Thus, due to its singularity at zero and its boundedness on every compact subset of $(0,1]$, there is $\overline{\Gamma}\in (0,1/2)$ small enough such that \begin{equation}
g \text{ is decreasing on } \left( 0,  \overline{\Gamma} \right) \;\;\text{ and }\;\; \left\{
\begin{array}{lcl}
g(\tilde{\Gamma})> g(\overline{\Gamma}) & \text{ for } & \tilde{\Gamma}\in \left( 0,\overline{\Gamma} \right)\ , \\
 & & \\
g(\tilde{\Gamma}) < g(\overline{\Gamma}) & \text{ for } & \tilde{\Gamma}\in \left( \overline{\Gamma},1 \right)\ ,
\end{array}
\right. \label{bach}
\end{equation}
and
\begin{equation}
\varrho \text{ is differentiable at } \overline{\Gamma} \;\;\;\text{ with }\;\;\; \overline{\Gamma} \varrho'(\overline{\Gamma}) - \varrho(\overline{\Gamma}) < 0 \ . \label{pachelbel}
\end{equation}
Setting $\overline{H} := H_c(\overline{\Gamma})$ the choice of $\overline{\Gamma}$ along with \eqref{eq_Gammat}, \eqref{bach}, \eqref{pachelbel}, and the definition of $H_c$ guarantees that 
\begin{equation} \label{eq_signf1}
f_1(\overline{H},\tilde{\Gamma}) <0 \ \text{for}\ \tilde{\Gamma} \in (0,\overline{\Gamma})\quad \text{and}\quad 
f_1(\overline{H},\tilde{\Gamma}) >0 \ \text{for}\ \tilde{\Gamma} \in (\overline{\Gamma},1)
\end{equation}
and, recalling also the monotonicity of $\phi,$  that $H_c$ is increasing on $(0,\overline{\Gamma})$. A further consequence of \eqref{eq_Gammat} and \eqref{eq_signf1} is that
\begin{equation}
\overline{H} < H_c(\tilde{\Gamma}) \quad \text{ and } \quad f_1(\tilde{H},\tilde{\Gamma})>0 \quad \text{ for } \quad (\tilde{H},\tilde{\Gamma})\in [H_*,\overline{H}]\times (\overline{\Gamma},1)\ . \label{corelli}
\end{equation}
The above observations \eqref{eq_signf1} motivate 
to decompose $[H_*,2H_*] \times (0,1)$ according to 
$$
Q_1 = [H_*,\overline{H}] \times (0,\overline{\Gamma}]\ , \; 
Q_2 = [H_*,\overline{H}] \times (\overline{\Gamma},1)\ , \;  
Q_3 = (\overline{H},2H_*] \times (\overline{\Gamma},1)\ , \;  
Q_4 = (\overline{H},2H_*] \times (0,\overline{\Gamma}] \ ,
$$
see Figure~\ref{fig1}. Observe that, due to \eqref{eq_signf1} and the properties of $f_1$ and $f_2$, we have
\begin{eqnarray}
&& Q_1 \text{ is positively invariant for \eqref{app_d1}-\eqref{app_d2},}\label{eq_Q1}\\
&& \text{a trajectory of \eqref{app_d1}-\eqref{app_d2} can escape $Q_3$ only through $\overline{{Q}_3} \cap Q_4,$} \label{eq_Q3}\\
&&\text{a trajectory can escape of \eqref{app_d1}-\eqref{app_d2} $Q_4$ only through $Q_1 \cap \overline{{Q}_4}.$} \label{eq_Q4}
\end{eqnarray}

\begin{figure}
\centering\includegraphics[width=10cm]{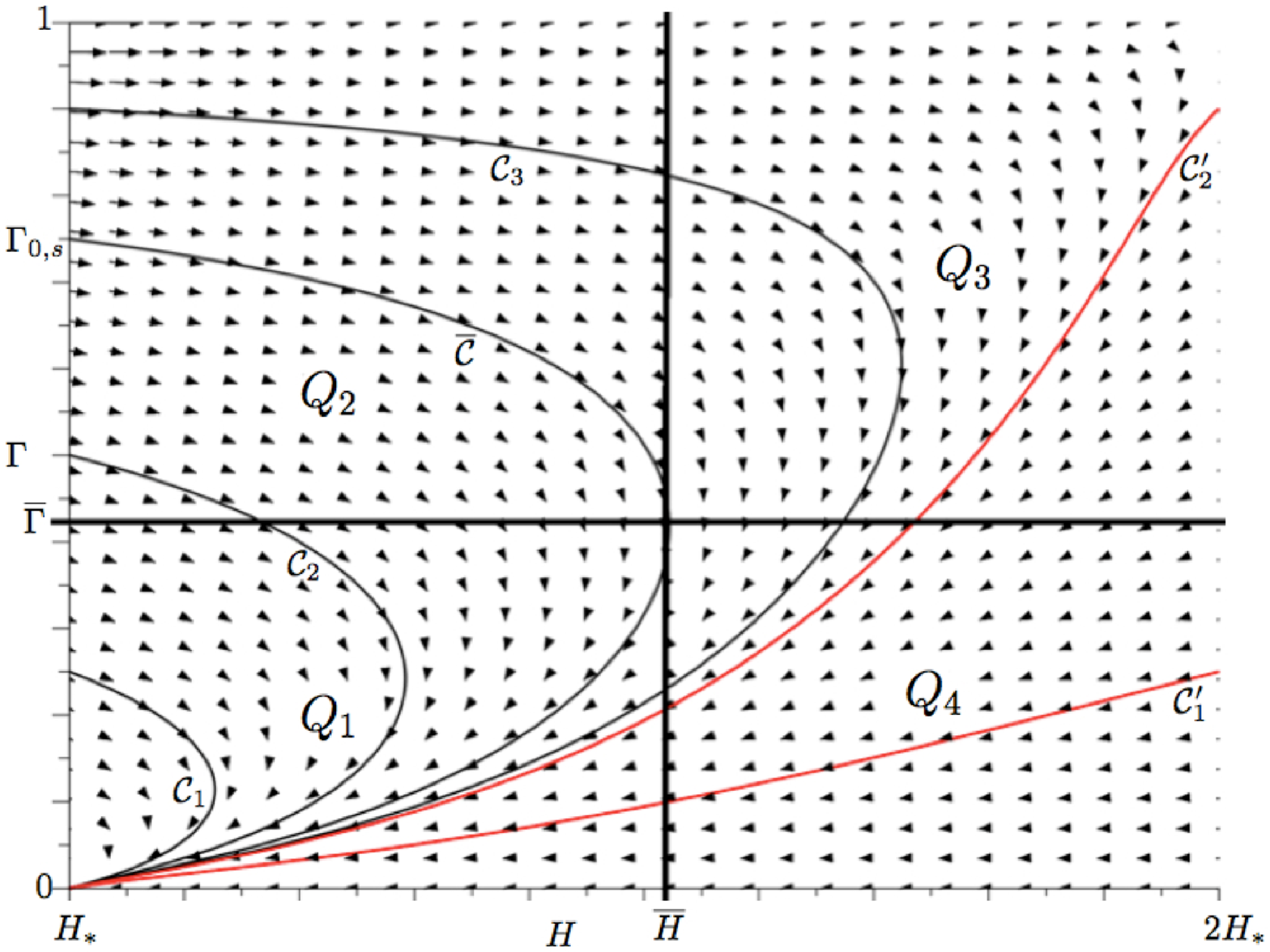}
%\centering\includegraphics[width=10cm]{portraitphase2.pdf}
\caption{\small Phase plane for \eqref{app_d1}-\eqref{app_d2}}\label{fig1}
\end{figure}
\medskip

We now study the properties of solutions to \eqref{app_d1}-\eqref{app_d2} with either $H_0=H_*$ or $H_0=2H_*$ in \eqref{app_d3}. A first step in that direction is the following:

%%%%%%%%%%%%%%%%%%%%%%%%%%%%%%%%%%%%%%%%%%%%%%%%%%%%%%%%%%%%%%%%%%%%
\begin{lemma} \label{le.a1}
There exists a unique $\Gamma_{0,s}\in (\overline{\Gamma} , 1)$ such that the non-extendable solution $(H_s,\Gamma_s)$ to \eqref{app_d1}-\eqref{app_d2} with initial condition $(H_s,\Gamma_s)(0)=(H_*,\Gamma_{0,s})$ satisfies 
$$
\left( \overline{H} , \overline{\Gamma} \right) \in \mathcal{C}_l(\Gamma_{0,s}) := \left\{ (H_s,\Gamma_s)(\xi)\ : \ \xi\ge 0 \right\}\ .
$$
\end{lemma} 
%%%%%%%%%%%%%%%%%%%%%%%%%%%%%%%%%%%%%%%%%%%%%%%%%%%%%%%%%%%%%%%%%%%%

\begin{proof}
Let $(H,\Gamma)$ be the non-extendable solution to \eqref{app_d1}-\eqref{app_d2} with initial condition $(H,\Gamma)(0) = \left( \overline{H} , \overline{\Gamma} \right)$ defined on $(-\xi_\alpha,\xi_\omega)$. By Proposition~\ref{prop_app1}~(a) and Proposition~\ref{prop_step2}, we have $\xi_\omega=\infty$, 
$$
(H,\Gamma)(\xi) \in (H_*,2H_*)\times (0,1)\ , \quad \xi\in (0,\infty)\ , \quad \text{ and } \quad \lim_{\xi\to\infty} (H,\Gamma)(\xi) = (H_*,0)\ .
$$
Furthermore, differentiating \eqref{app_d1}, we infer from \eqref{pachelbel} and the properties
$$
6 \overline{\Gamma}\, \overline{H} (2H_* - \overline{H}) - 12D (\overline{H}-H_*) \varrho(\overline{\Gamma}) = H'(0) = 0
$$
stemming from the definition of $\left( \overline{H} , \overline{\Gamma} \right)$ that
$$
H''(0) = \frac{12D (\overline{H} - H_*)}{G \overline{H}^3 \overline{\Gamma} \left( \overline{\Gamma}\ \overline{H} + 4 D \varrho(\overline{\Gamma}) \right)}\ \left( \varrho(\overline{\Gamma}) - \overline{\Gamma} \varrho'(\overline{\Gamma}) \right)\ \Gamma'(0)\ .
$$
Thanks to \eqref{pachelbel} and $\Gamma'(0)<0$, we conclude that $H''(0)<0$. Consequently, there is $\delta>0$ such that $H(\xi)<\overline{H}$ for $\xi\in (-\delta,\delta)\setminus\{0\}$. Recalling that $\Gamma$ is decreasing on $(-\xi_\alpha,\infty)$ by Proposition~\ref{prop_app1}~(a), there holds $\Gamma(-\xi)> \overline\Gamma >\Gamma(\xi)$ for $\xi\in (0,\delta)$ so that $(H,\Gamma)(\xi)\in Q_2$ for $\xi\in (-\delta,0)$ and $(H,\Gamma)(\xi)\in Q_1$ for $\xi\in (0,\delta)$. On the one hand the already mentioned positive invariance of $Q_1$ implies that $(H,\Gamma)(\xi)\in Q_1$ for $\xi\in (0,\infty)$. On the other hand, either $\xi_\alpha<\infty$ and we infer from Proposition~\ref{prop_app1}~(b) that there is $\tilde{\xi}_\alpha\in (-\xi_\alpha,0)$ such that $H(\tilde{\xi}_\alpha)=H_*$. We then put $\Gamma_{0,s} := \Gamma(\tilde{\xi}_\alpha)$ and $(H_s,\Gamma_s)(\xi) := (H,\Gamma)(\xi+\tilde{\xi}_\alpha)$ for $\xi\in (-\xi_\alpha - \tilde{\xi}_\alpha, \infty)$ to complete the proof. Or $\xi_\alpha=\infty$ and, according to the definition of $\delta$, 
$$
\xi_2 := \inf\left\{ \xi < 0\ : \ (H,\Gamma)(\zeta)\in (H_*,\overline{H})\times (\overline{\Gamma},1) \;\;\text{ for all }\;\; \zeta \in (\xi,0) \right\} \le - \delta\ .
$$
By \eqref{app_d1}, \eqref{app_d2}, and \eqref{corelli},
$$
H'(\xi)>0 \quad\text{ and }\quad \Gamma'(\xi)<0 \ , \quad \xi\in (\xi_2,0)\ .
$$
In particular, $H(\xi)\in (H_*,\overline{H})$ for all $\xi\in (\xi_2,0)$ and Proposition~\ref{prglobal} implies that $\xi_2$ is finite. As $\xi_\alpha=\infty$, we necessarily have  $\Gamma(\xi_2)<1$ and thus $H(\xi_2)=H_*$. We then set $\Gamma_{0,s} := \Gamma(\xi_2)$ and complete the proof as in the previous case. 
\end{proof}

Having identified the trajectory passing through $\left( \overline{H} , \overline{\Gamma} \right)$ (which corresponds to the curve $\bar{\mathcal{C}}$ in Figure~\ref{fig1}) we are in a position to classify the behavior of the solutions to \eqref{app_d1}-\eqref{app_d3} emanating from the vertical sides of the rectangle $(H_*,2H_*)\times (0,1)$. 

%%%%%%%%%%%%%%%%%%%%%%%%%%%%%%%%%%%%%%%%%%%%%%%%%%%%%%%%%%%%%%%%%%%%
\begin{lemma}\label{le.a2} 
Given $\Gamma_0 \in (0,1),$ let $(H,\Gamma)$ be the unique non-extendable solution to \eqref{app_d1}-\eqref{app_d2} with initial condition $(H_*,\Gamma_0),$ which is defined for all $\xi\in [0,\infty)$ and connects $(H_*,\Gamma_0)$  to $(H_*,0)$ according to Proposition~\ref{prop_app1} and Proposition~\ref{prop_step2}. Introducing 
$$
\mathcal{C}_l(\Gamma_0) := \left\{ (H,\Gamma)(\xi)\ :\ \xi\ge 0 \right\}\ ,
$$
the courses of $\mathcal{C}_l(\Gamma_0)$ are the following:
\begin{itemize}
\item[\bf (L1)] If $\Gamma_0\in (0,\overline{\Gamma}]$, then $\mathcal{C}_l(\Gamma_0)$ remains in $Q_1$, see the trajectory $\mathcal{C}_1$ in Figure~\ref{fig1}.
\item[\bf (L2)] If $\Gamma_0\in (\overline{\Gamma}, \Gamma_{0,s}]$, then $\mathcal{C}_l(\Gamma_0)$ passes successively through $Q_2$ and then through $Q_1$, see the trajectory $\mathcal{C}_2$ in Figure~\ref{fig1}.
\item[\bf (L3)] If $\Gamma_0\in (\Gamma_{0,s},1)$, then $\mathcal{C}_l(\Gamma_0)$ passes successively through $Q_2$, $Q_3$,  $Q_4$, and finally through $Q_1$, see the trajectory $\mathcal{C}_3$ in Figure~\ref{fig1}.  
\end{itemize}
\end{lemma}
%%%%%%%%%%%%%%%%%%%%%%%%%%%%%%%%%%%%%%%%%%%%%%%%%%%%%%%%%%%%%%%%%%%%

\begin{proof}
\textbf{(L1)}: The first assertion follows from \eqref{eq_Q1}:  $Q_1$ is positively invariant for \eqref{app_d1}-\eqref{app_d2}.

\smallskip

\noindent \textbf{(L2)}: If $\Gamma_0=\Gamma_{0,s}$, then $\mathcal{C}_l(\Gamma_{0,s})$ passes successively through $Q_2$ and then through $Q_1$ by Lemma~\ref{le.a1}. If $\Gamma_0\in (\overline{\Gamma},\Gamma_{0,s})$, we infer from the monotonicity of $\Gamma$ and the Cauchy-Lipschitz theorem that $\mathcal{C}_l(\Gamma_0) \cap Q_2$ stays below $\mathcal{C}_l(\Gamma_{0,s})\cap Q_2$ and must cross $Q_1\cap \overline{Q_2}$ at some $\xi>0$ with $H(\xi)\in (H_*,\overline{H})$ and then enter $Q_1$.  We complete the proof with the help of the positive invariance of $Q_1$. 

\smallskip

\noindent \textbf{(L3)}: If $\Gamma_0\in (\Gamma_{0,s},1)$ then $\mathcal{C}_l(\Gamma_0) \cap Q_2$ lies above $\mathcal{C}_l(\Gamma_{0,s})\cap Q_2$ and thus enters $Q_3$. Since $(H,\Gamma)(\xi)\to (H_*,0)$ as $\xi\to \infty$ by Proposition~\ref{prop_step2} and $(H,\Gamma)(\xi)\ne (\overline{H},\overline{\Gamma})$ for all $\xi\ge 0$ by Lemma~\ref{le.a1}, it follows from \eqref{eq_Q3} and \eqref{eq_Q4} that the curve $\mathcal{C}_l(\Gamma_0)$ passes from $Q_3$ to $Q_4$ and finally from $Q_4$ to $Q_1$.
\end{proof}

\medskip

If $\Gamma_0\in (\Gamma_{0,s},1)$, then the curve  $\mathcal{C}_l(\Gamma_0)$ satisfies the alternative \textbf{(L3)} by Lemma~\ref{le.a2} and thus intersects $\overline{Q_4} \cap Q_1$ in exactly one point $(\overline{H},X_l(\Gamma_0))$ with $X_l(\Gamma_0)\in (0,\overline{\Gamma}).$ We also set $X_l(\Gamma_{0,s}) = \overline{\Gamma}$ and define
\begin{equation}
\mathcal{X}_l := \left\{ X_l(\Gamma_0)\ :\ \Gamma_0 \in  [\Gamma_{0,s},1) \right\}\ , \quad X_l := \inf \mathcal{X}_l \ . \label{vivaldi}
\end{equation}

%%%%%%%%%%%%%%%%%%%%%%%%%%%%%%%%%%%%%%%%%%%%%%%%%%%%%%%%%%%%%%%%%%%%
\begin{lemma}\label{le.a3}
Given $\Gamma_0 \in (0,1),$ let $(H,\Gamma)$ be the unique non-extendable solution to \eqref{app_d1}-\eqref{app_d2} with initial condition $(2H_*,\Gamma_0),$ which is defined for all $\xi\in [0,\infty)$ and connects $(2H_*,\Gamma_0)$  to $(H_*,0)$ according to Proposition~\ref{prop_app1} and Proposition~\ref{prop_step2}. Introducing 
$$
\mathcal{C}_r(\Gamma_0) := \left\{ (H,\Gamma)(\xi)\ :\ \xi\ge 0 \right\}\ ,
$$
the courses of $\mathcal{C}_r(\Gamma_0)$ are the following: 
\begin{itemize}
\item[\bf (R1)] If $\Gamma_0\in (0,\overline{\Gamma}]$, then $\mathcal C_r(\Gamma_0)$ passes through $Q_4$ and then through $Q_1$, see the trajectory $\mathcal{C}'_1$ in Figure~\ref{fig1}. 
\item[\bf (R2)] If $\Gamma_0\in (\overline{\Gamma},1)$, then $\mathcal C_r(\Gamma_0)$ passes through $Q_3,$ $Q_4,$  and finally through $Q_1$, see the trajectory $\mathcal{C}'_2$ in Figure~\ref{fig1}.
\end{itemize}
\end{lemma}
%%%%%%%%%%%%%%%%%%%%%%%%%%%%%%%%%%%%%%%%%%%%%%%%%%%%%%%%%%%%%%%%%%%%

\begin{proof}
\textbf{(R1)}: If $\Gamma_0\in (0,\overline{\Gamma}]$, then $(H,\Gamma)(0)\in Q_4$ and we infer from the monotonicity of $\Gamma$ stated in Proposition~\ref{prop_app1}~(a) and the behavior of the trajectory as $\xi\to\infty$ described in Proposition~\ref{prop_step2} that $\mathcal{C}_r(\Gamma_0)$ remains in $Q_4$ until it crosses $\overline{Q_4}\cap Q_1$ and enters in $Q_1$. It subsequently remains in $Q_1$ owing to the positive invariance of $Q_1$.  

\smallskip

\noindent \textbf{(R2)}: If $\Gamma_0\in (\overline{\Gamma},1)$, then $(H,\Gamma)(0)\in Q_3$ and the behavior for large $\xi$ of $(H,\Gamma)$ implies that it has to enter in $Q_1$. Owing to \eqref{eq_Q3}, the trajectory leaves $Q_3$ through $\overline{Q_3}\cap Q_4$. It then goes through $Q_4$ before entering $Q_1$ and staying there as in the previous alternative.
\end{proof}

\medskip

A consequence of Lemma~\ref{le.a3} is that, for each $\Gamma_0\in (0,1)$, the curve $\mathcal{C}_r(\Gamma_0)$ intersects $\overline{Q_4}\cap Q_1$ in exactly one point $(\overline{H},X_r(\Gamma_0))$ with $X_r(\Gamma_0)\in (0,\overline{\Gamma}).$ We then define
\begin{equation}
\mathcal{X}_r := \left\{ X_r(\Gamma_0)\ :\ \Gamma_0 \in (0,1) \right\}\ , \quad X_r := \sup \mathcal{X}_r \ . \label{haendel}
\end{equation}

\medskip

We now derive further properties of $X_l$ (defined in \eqref{vivaldi}) and $X_r$ (defined in \eqref{haendel}). 

%%%%%%%%%%%%%%%%%%%%%%%%%%%%%%%%%%%%%%%%%%%%%%%%%%%%%%%%%%%%%%%%%%%%
\begin{lemma} \label{lem_step3}
The following assertions hold true:
\begin{itemize}
\item[(i)] $X_l \notin \mathcal X_l$ and $X_l < \overline{\Gamma}$,
\item[(ii)] $X_r \notin \mathcal X_r$ and $X_r>0$,
\item[(iii)] $0 < X_r \le  X_l < \overline{\Gamma}.$
\end{itemize}
\end{lemma}
%%%%%%%%%%%%%%%%%%%%%%%%%%%%%%%%%%%%%%%%%%%%%%%%%%%%%%%%%%%%%%%%%%%%

\begin{proof}
(i): Since $X_l(\Gamma_{0,s})=\overline{\Gamma}$, we have $X_l\in [0,\overline{\Gamma}]$ and $\overline{\Gamma}\in \mathcal{X}_l$. Assume now for contradiction that $X_l\in \mathcal{X}_l.$ Then there exists $\Gamma_0 \in [\Gamma_{0,s},1)$ such that $X_l= X_l(\Gamma_0)$. Pick $\Gamma_1 > \Gamma_0$. On the one hand, the Cauchy-Lipschitz theorem guarantees that $\mathcal{C}_l(\Gamma_0) \cap \mathcal{C}_l(\Gamma_1)=\emptyset.$ On the other hand, $\mathcal{C}_l(\Gamma_1)$ satisfies \textbf{(L3)} by Lemma~\ref{le.a2} and $\mathcal{C}_l(\Gamma_1)$ lies below $\mathcal{C}_l(\Gamma_0)$ when both cross $\overline{Q_4}\cap Q_1$. This means that $X_l(\Gamma_1) < X_l(\Gamma_0) = X_l$ and contradicts the definition of $X_l$. Consequently, $X_l\in \mathcal{X}_l$ which in turn implies that $X_l<\overline{\Gamma}$ (as $\overline{\Gamma}\in \mathcal{X}_l$). 

\smallskip

\noindent (ii): A similar argument shows that $X_r \notin \mathcal{X}_r.$ Observing that $X_r(\Gamma_0)>0$ for all $\Gamma_0\in (0,1)$, we readily conclude that $X_r>0.$  

\smallskip

\noindent (iii): To prove the last assertion, we realize that, given $\Gamma_0\in [\Gamma_{0,s},1)$ and $\Gamma_1 \in (0,1)$, the curve $\mathcal{C}_r(\Gamma_1)$ cannot cross $\overline{Q_4}\cap Q_1$ above $\mathcal{C}_l(\Gamma_0)$ without having crossed $\mathcal{C}_l(\Gamma_0)$ previously. Since $\mathcal{C}_r(\Gamma_1)\cap \mathcal{C}_l(\Gamma_0) = \emptyset$ by the Cauchy-Lipschitz theorem, this yields $X_l(\Gamma_0) > X_r({\Gamma}_1),$ and consequently $X_r \leq X_l.$
\end{proof}

We are now prepared to prove the following result, which guarantees the existence of a global solution of \eqref{app_d1}-\eqref{app_d2} in $(H_*,2H_*) \times (0,1)$.

%%%%%%%%%%%%%%%%%%%%%%%%%%%%%%%%%%%%%%%%%%%%%%%%%%%%%%%%%%%%%%%%%%%%
\begin{proposition} \label{prop_step3}
There exists a global solution $(H,\Gamma)$ to \eqref{app_d1}-\eqref{app_d2} such that 
$$
(H,\Gamma)(\xi) \in (H_*,2H_*) \times (0,1) \ , \quad \xi \in \R\,.
$$
\end{proposition}
%%%%%%%%%%%%%%%%%%%%%%%%%%%%%%%%%%%%%%%%%%%%%%%%%%%%%%%%%%%%%%%%%%%%

\begin{proof}
Invoking Lemma~\ref{lem_step3}, we see that $\mathcal{J}:=[X_r,X_l]$ is a non-empty compact interval contained in $(0,1)$. Pick $\Gamma_*\in \mathcal{J}$ and let $(H,\Gamma)$ be the non-extendable solution to \eqref{app_d1}-\eqref{app_d2} with initial condition $(\overline{H},\Gamma_*)$ which is defined on $(-\xi_{\alpha},\infty)$ according to Proposition~\ref{prop_app1}~(a). Introducing the corresponding trajectory 
$$
\mathcal{C} := \left\{ (H,\Gamma)(\xi)\ :\ \xi\in (-\xi_\alpha,\infty) \right\}\ ,
$$
we infer from the fact that $\mathcal C$ passes through $Q_3,$ Lemma~\ref{le.a2}, Lemma~\ref{le.a3}, and Lemma~\ref{lem_step3} that $\Gamma_*< X_l(\Gamma_0)$ for all $\Gamma_0\in (\Gamma_{0,s},1)$ and $\Gamma_*>X_r(\Gamma_0)$ for all $\Gamma_0\in (0,1)$, so that $\mathcal{C}$ cannot cross the vertical sides of the rectangle $(H_*,2H_*) \times (0,1)$.  Consequently, $(H,\Gamma)(\xi) \in (H_*,2H_*) \times (0,1)$ for all $\xi \in (-\xi_{\alpha},\infty).$  Proposition~\ref{prop_app1} then entails that $\xi_{\alpha} = \infty $ while Proposition~\ref{prop_step2} and Proposition~\ref{prglobal} ensure that $(H,\Gamma)$ connects $(2H_*,1)$ to $(H_*,0)$.  We have thus constructed a global solution to \eqref{app_d1}-\eqref{app_d2} with the desired properties.
\end{proof}

\medskip

Theorem~\ref{T1}  and Proposition~\ref{P222} for $G>0$ and $D>0$ are now an immediate consequence of Proposition~\ref{prop_step3}.

%%%%%%%%%%%%%%%%%%%%%%%%%%%%%%%%%%%%%%%%%%%%%%%%%%%%%%%%%%%%%%%%%%%%
\begin{remark} {\rm A shorter proof for the existence of a traveling wave solution may be obtained when assumption~(\ref{assumrho}) is strengthened: If there is a $C^1$-smooth extension of $\varrho$ to $\mathbb{R}$ with  $\varrho'(1)<0,$  the principle of linearized stability implies that $(H_*,0)$ is a sink and that $(2H_*,1)$ is a saddle point of \eqref{app_d1}-\eqref{app_d2}. Moreover, the stable manifold of $(2H_*,1)$ points into the invariant rectangle $[H_*,2H_*]\times [0,1]$. Invoking the theorem of Grobman-Hartman, see \cite[Theorems~19.9 \&~19.11]{Am90} for instance, this implies the existence of a heteroclinic orbit connecting $(H_*,0)$ with 
$(2H_*,1)$. }
\end{remark}
%%%%%%%%%%%%%%%%%%%%%%%%%%%%%%%%%%%%%%%%%%%%%%%%%%%%%%%%%%%%%%%%%%%%

%%%%%%%%%%%%%%%%%%%%%%%%%%%%%%%%%%%%%%%%%%%%%%%%%%%%%%%%%%%%%%%%%%%%
%%%%%%%%%%%%%%%%%%%%%%%%%%%%%%%%%%%%%%%%%%%%%%%%%%%%%%%%%%%%%%%%%%%%
\section{Stationary Solutions}\label{stat_sol}
%%%%%%%%%%%%%%%%%%%%%%%%%%%%%%%%%%%%%%%%%%%%%%%%%%%%%%%%%%%%%%%%%%%%
%%%%%%%%%%%%%%%%%%%%%%%%%%%%%%%%%%%%%%%%%%%%%%%%%%%%%%%%%%%%%%%%%%%%

We briefly discuss stationary solutions corresponding to the choice $c=0$ in \eqref{i5}-\eqref{i6} and leading to the system
\begin{align}
\frac{G H^3}{3} H' - \frac{H^2}{2} \sigma'(\Gamma) \Gamma' & = 0\ , &&\xi\in \RR\ , \label{i5c} \\
\frac{G H^2}{2} \Gamma H' - \left[ H \Gamma \sigma'(\Gamma) - D \right] \Gamma' & =  0 \ ,  &&\xi\in \RR\ , \label{i6c}\\
H\ge 0\ ,\qquad 0\le \Gamma &< 1\ , &&\xi\in \RR\ . \label{i7c}
\end{align}
Clearly, taking $H$ as any positive constant and $\Gamma$ as any constant in $[0,1)$ gives a stationary solution.

\medskip

We now distinguish several cases. 

\smallskip

\noindent \textit{Case~$G=D=0$}. The system \eqref{i5c}-\eqref{i6c} reduces to $H^2 \sigma'(\Gamma) \Gamma' = 0$. The only conclusions we can draw from this are that $\Gamma$ is constant on connected components of the positivity set of $H$ (without further restrictions on $H$) while there is no constraint on $\Gamma$ on the zero set of $H$. 

\smallskip

\noindent \textit{Case $G=0, D>0$}. The non-positivity of $\sigma'$ implies that $\Gamma'$ vanishes identically, that is, $\Gamma$ is constant while there is no restriction on $H$.

\smallskip

\noindent \textit{Case $G>0, D=0$}. There is a wide variety of discontinuous solutions. However, the only continuous solutions with $H\not\equiv 0$ are constant. Indeed, multiplying \eqref{i5c} by $3\Gamma$ and \eqref{i6c} by $-2H$ and adding the results yield $H^2 \Gamma \sigma'(\Gamma) \Gamma' = 0$. Thus, on connected components of the positivity set of $\Gamma$ we see from \eqref{i5c} that $H$ has to be constant and the same conclusion holds true on connected components of the zero set of $\Gamma$. Since $H\not\equiv 0$ is assumed to be continuous, it is necessarily a positive constant which, together with \eqref{i6c}, implies $\Gamma$ is a constant.

\smallskip

\noindent \textit{Case $G>0, D>0$}. The same computations as in the previous case yield $H [H \Gamma \sigma'(\Gamma) \Gamma' - 4 D] \Gamma'=0$ from which we deduce that $H \Gamma'=0$. In combination with \eqref{i5c} this implies that $H$ and $\Gamma$ are constants on connected components of the positivity set of $H$, while \eqref{i6c} ensures that $\Gamma$ is constant on the connected components of the zero set of $H$. Thus, in this case again, the only continuous solutions are the constants.

%%%%%%%%%%%%%%%%%%%%%%%%%%%%%%%%%%%%%%%%%%%%%%%%%%%%%%%%%%%%%%%%%%%%
%%%%%%%%%%%%%%%%%%%%%%%%%%%%%%%%%%%%%%%%%%%%%%%%%%%%%%%%%%%%%%%%%%%%
\section*{Acknowledgments}
%%%%%%%%%%%%%%%%%%%%%%%%%%%%%%%%%%%%%%%%%%%%%%%%%%%%%%%%%%%%%%%%%%%%
%%%%%%%%%%%%%%%%%%%%%%%%%%%%%%%%%%%%%%%%%%%%%%%%%%%%%%%%%%%%%%%%%%%%

The work of Ph.L. was partially supported by the Deutscher Akademischer Austausch Dienst (DAAD) while enjoying the hospitality of the Institut f\"ur Angewandte Mathematik, Leibniz Universit\"at Hannover.

%%%%%%%%%%%%%%%%%%%%%%%%%%%%%%%%%%%%%%%%%%%%%%%%%%%%%%%%%%%%%%%%%%%%
%%%%%%%%%%%%%%%%%%%%%%%%%%%%%%%%%%%%%%%%%%%%%%%%%%%%%%%%%%%%%%%%%%%%
\bibliographystyle{amsplain}
%

%
%%%%%%%%%%%%%%%%%%%%%%%%%%%%%%%%%%%%%%%%%%%%%%%%%%%%%%%%%%%%%%%%%%%%
%%%%%%%%%%%%%%%%%%%%%%%%%%%%%%%%%%%%%%%%%%%%%%%%%%%%%%%%%%%%%%%%%%%%
%
\end{document}